\documentclass[12pt,twoside,reqno]{amsart}

\usepackage{color}
\usepackage{graphicx}
\usepackage{amsmath}
\usepackage{amssymb}
\usepackage{amsfonts}
\usepackage{multicol}
\usepackage[ansinew]{inputenc}
\usepackage{amscd}
\usepackage[mathscr]{eucal}
\usepackage{hyperref}
\usepackage{url}

\newcommand{\R}{\mathbb{R}}

\newcommand{\N}{\mathbb{N}}
\newcommand{\Z}{\mathbb{Z}}

\newcommand{\SL}{{\rm SL}}

\newcommand{\GL}{\rm GL}

\newcommand{\Sym}{{\rm Sym}}
\newcommand{\Mat}{{\rm Mat}}

\newcommand{\V}{\mathscr{V}}
\newcommand{\Fscr}{\mathscr{F}}
\newcommand{\Escr}{\mathscr{E}}

\newcommand{\abs}[1]{\bigl| #1 \bigr|} 
\newcommand{\norm}[1]{\lVert#1\rVert} 
\newcommand{\normtwo}[1]{
	{\left\vert\kern-0.25ex\left\vert\kern-0.25ex\left\vert #1
		\right\vert\kern-0.25ex\right\vert\kern-0.25ex\right\vert} }

\newcommand{\Proj}{\mathbb{P}(\R^ 2)}

\newcommand{\Gr}{{\rm Gr}}
\newcommand{\Pp}{\mathbb{P}}

\newcommand{\hatv}{\hat{v}}

\newcommand{\Projm}{\mathbb{P}(\R^m)}

\newcommand{\Bscr}{\mathscr{B}}

\newcommand\restr[2]{{
		\left.\kern-\nulldelimiterspace 
		#1 
		\vphantom{\big|} 
		\right|_{#2} 
}}

\newcommand{\Mscr}{\mathscr{M}}
\newcommand{\Mcal}{\mathcal{M}^p_C}

\newcommand{\Scal}{\mathcal{S}}	
\newcommand{\Lip}{\mathrm{Lip}}
\newcommand{\Hol}{\mathscr{H}_p}
\newcommand{\Holal}{\mathscr{H}_{\alpha}}
\newcommand{\Sp}{\mathrm{Sp}}

\theoremstyle{plain}
\newtheorem{theorem}{Theorem}[section]
\newtheorem{proposition}{Proposition}[section]
\newtheorem{corollary}[proposition]{Corollary}
\newtheorem{lemma}[proposition]{Lemma}

\theoremstyle{definition}
\newtheorem{definition}{Definition}[section]

\theoremstyle{definition}
\newtheorem{remark}{Remark}[section]

\numberwithin{equation}{section}

\newcommand{\Prob}{\mathrm{Prob}}

\newcommand{\EE}{\mathbb{E}}

\newcommand{\supp}{\mathrm{supp}}
\newcommand{\Qop}{\mathcal{Q}}
\newcommand{\Rop}{\mathcal{R}}
\newcommand{\Mea}{\mathrm{Meas}}

\title[Non-invertible and Non-compact random cocycles] {H\"older continuity of Lyapunov exponents for non-invertible and non-compact random cocycles}

\begin{document}

\author[P. Duarte]{Pedro Duarte}
\address{Departamento de Matem\'atica and CEMS.UL\\
Faculdade de Ci\^encias\\
Universidade de Lisboa\\
Portugal
}
\email{pmduarte@fc.ul.pt}

\author[T. Graxinha]{Tom\'e  Graxinha}
\address{Departamento de Matem\'atica and CEMS.UL\\
	Faculdade de Ci\^encias\\
	Universidade de Lisboa\\
	Portugal}
\email{graxinhatome@hotmail.com }

\begin{abstract}


We study the regularity of Lyapunov exponents for random linear cocycles taking values in $\Mat_m(\R)$ and driven by i.i.d.\ processes. Under three natural conditions—finite exponential moments, a spectral gap between the top two Lyapunov exponents, and quasi-irreducibility of the associated semigroup—we prove that the top Lyapunov exponent is H\"older continuous with respect to the Wasserstein distance.
In the final section, we apply the main results to Schr\"odinger cocycles with unbounded potentials.
\end{abstract}

\maketitle

\tableofcontents

\section{Introduction}\label{intro}

A \emph{linear cocycle} is a bundle map $F:X \times \R^m \to X \times \R^m$ of the form
	\[
	F(\omega, v) = (f(\omega), A(\omega)\, v),
	\]
	where $f:X \to X$ is the \emph{base map}, and $A:X \to \Mat_m(\R)$ governs the fiber dynamics. The $n$-th iterate is
	\[
	F^n(\omega, v) = (f^n(\omega), A^n(\omega)\, v), \quad A^n(\omega) := A(f^{n-1}(\omega)) \cdots A(\omega).
	\]
	The \emph{Lyapunov exponent} at $\omega \in X$ along $v \in \R^m$ is defined by
	\[
	L(F;\omega, v) := \lim_{n \to \infty} \frac{1}{n} \log \norm{A^n(\omega) v},
	\]
	when the limit exists. The \emph{top Lyapunov exponent} is
	\[
	L_1(F;\omega) := \lim_{n \to \infty} \frac{1}{n} \log \norm{A^n(\omega)}.
	\]
	
	Under a probability measure $\mu$ preserved by $f$, Furstenberg and Kesten~\cite{FK60} proved the almost sure existence and finiteness of $L_1$ under the \emph{finite moment condition}:
	\[
	\int \log \norm{A(\omega)^{\pm 1}}\, d\mu(\omega) < \infty.
	\]
	Oseledets' theorem~\cite{Ose68} extends this, yielding a full \emph{Lyapunov spectrum} of exponents $L_1 > L_2 > \cdots$ that are invariant under $f$ and almost surely constant if $f$ is ergodic. The second exponent is given by
	\[
	L_2(F;\omega) := \lim_{n \to \infty} \frac{1}{n} \log s_2(A^n(\omega)),
	\]
	where $s_2$ is the second singular value. More generally, exterior powers reduce the study of all exponents to the top exponent of suitable induced cocycles.
	
	Lyapunov exponents are fundamental in smooth dynamics and ergodic theory, encoding sensitivity to initial conditions and connecting to entropy. In mathematical physics, they play a key role in determining the spectral nature of Schr\"odinger operators.

 \bigskip

 In~\cite{Ru82}, Ruelle proved that projectively uniformly hyperbolic cocycles\footnote{cocycles with dominated splitting and one-dimensional expanding direction}  have analytically dependent top Lyapunov exponents.
 
 Beyond this, two main frameworks have emerged: the \emph{random setting} (e.g., Bernoulli or Markov shifts, or hyperbolic maps) and the \emph{quasi-periodic setting} (e.g., torus translations). This paper focuses on the random setting, particularly on the regularity of Lyapunov exponents.
 
 A probability measure $\mu$ on $\SL_m(\R)$, $\GL_m(\R)$, or $\Mat_m(\R)$ defines a random cocycle
 \[
 F(\omega, v) := (\sigma \omega, A(\omega) v),
 \]
 over the Bernoulli shift $\sigma$ on $X = \Mat_m(\R)^\N$, with $A(\omega) = \omega_0$. The corresponding top Lyapunov exponents are denoted $L_1(\mu), L_2(\mu)$.
 
 Furstenberg~\cite{Fur63} established a formula and positivity criterion for $L_1(\mu)$ in the $\SL_m(\R)$ case. Furstenberg and Kifer~\cite{FK83} later proved its generic continuity under  uniform moment bounds. Continuity for $\GL_2(\R)$ cocycles with compact support was shown by Bocker-Neto and Viana ~\cite{BV}, extended to Markov shifts by Malheiro and Viana~\cite{Viana-M}, and generalized to fiber-bunched cocycles by Backes, Brown and Butler~\cite{BBB}. For random $\GL_m(\R)$ cocycles, continuity was proved by Avila, Eskin and Viana~\cite{AEV}, and upper semi-continuity for cocycles with unbounded support by S\'anchez and Viana~\cite{SV2020}. See also~\cite{Jamerson-M, BarrientosMalicet2025} for further developments.

 Regarding quantitative continuity, 
 Le Page~\cite{LP89} proved H\"older continuity of $L_1$ under finite exponential moments, irreducibility, and a spectral gap between the first and second Lyapunov exponents. Extensions to Markov shifts appear in~\cite{DK-book, CDKM22}. For finitely supported $\GL_2(\R)$ cocycles with a  a spectral gap, weak-H\"older regularity was proved by Duarte and Klein~\cite{DK-Holder}. Still in the $\GL_2(\R)$ case, the minimum pointwise  log-H\"older continuity was established by Tall and Viana~\cite{Tall-Viana}. H\"older regularity under generic conditions has also been shown by Poletti et al.~\cite{DKP2022}  for fiber bunched $\GL_m(\R)$ cocycles, and by Barrientos and Malicet~\cite{BarrientosMalicet2025}  for random compositions of mostly contracting Lipschitz maps.
 
 Peres~\cite{Pe91} showed analyticity of $L_1$ with respect to the probability weights, for finitely supported $\GL_m(\R)$ cocycles; see also~\cite{BST2023, ADM25} for recent extensions.
 
 The optimality of Le Page's result is illustrated by Halperin's example~\cite{Halperin67}, streamlined by Simon and Taylor~\cite{ST85}, which shows that $L_1$ need not be smoother than H\"older. The necessity of Le Page's irreducibility and spectral gap assumptions is demonstrated by Klein et al.~\cite{DKS19}. Further bounds on H\"older exponents appear in~\cite{BeDu23, BCDFK}, bounding the regularity in terms of entropy and Lyapunov exponents.
 
 For non-invertible cocycles, stability and continuity of Lyapunov data under perturbations were studied by Froyland et al.~\cite{FLQ2010, DF2018}. Freijo and Duarte~\cite{DF2024} analyzed continuity for constant-rank cocycles, and Klein et  al.~\cite{DDKG25} established a dichotomy between analyticity and discontinuity in the finitely supported $\Mat_2(\R)$ setting.

 \bigskip

The main result of this paper (Theorem~\ref{thm non invertible Holder LE }) is the local H\"older continuity of the top Lyapunov exponent for random cocycles generated by probability measures on $\Mat_m(\R)$, under three key assumptions:
	(a) finite exponential moments;
	(b) quasi-irreducibility of the generated semigroup; and
	(c) a gap between the first two Lyapunov exponents.
	
	Assumption (a) also appears in Le Page~\cite{LP89}; its necessity will be illustrated in future work. The main novelties are the treatment of non-invertible and non-compact matrix cocycles and the continuity result in the Wasserstein topology -- contrasting with Le Page's parametric setting~\cite{LP89}.
	
	H\"older continuity of higher-order exponents follows via exterior powers (Corollaries~\ref{cont1} and~\ref{cont2}). As another consequence, we show H\"older continuity of the top exponent with respect to the $L^1$ distance for cocycles over Bernoulli shifts (Corollary~\ref{thm L1 Holder LE}), a result independently obtained by Barrientos and Malicet~\cite{BarrientosMalicet2025}.
	
	Finally, we obtain stable exponential-type large deviation estimates (Theorem~\ref{thm large deviation estimates}) as a byproduct of our method.

\bigskip

Our approach follows the classical spectral method of S.~V.~Nagaev~\cite{Nagaev}, adapted to random cocycles by Bougerol~\cite{Bou} and Le Page~\cite{LP82,LP89}. Le Page first used this approach to prove H\"older continuity of the top Lyapunov exponent, later simplified by Duarte and Baraviera~\cite{BD}.
	
	The key step combines quasi-irreducibility with a spectral gap to establish uniform convergence of the expected logarithmic growth:
	\begin{equation}
		\label{Expected log growth rate}
		\lim_{n\to \infty} \EE\left( \frac{1}{n} \log \norm{A^n v} \right) = L_1(\mu),
	\end{equation}
	uniformly over unit vectors $v \in \R^m$. This implies uniform ergodicity of the associated Markov operator: for suitable observables $\phi$ and $\mu$-stationary measures $\eta$, the iterates $\Qop_\mu^n \phi$ converge uniformly to $\int \phi\, d\eta$. Following~\cite{BD}, this yields H\"older regularity.
	
	Non-invertibility introduces additional complications: the projective action $\hat{g} \hat{v}$ is undefined when $v$ lies in $\ker(g)$, obstructing the extension of Furstenberg's theory. A minimal requirement is the finite moment condition:
	\[
	\sup_{\norm{v} = 1} \int \abs{ \log \norm{g v} }\, d\mu(g) < \infty,
	\]
	ensuring $g v \neq 0$ for $\mu$-almost every $g$ and all $v$. Otherwise, discontinuities may occur and must be handled separately as in~\cite{DDKG25}. For H\"older continuity, stronger finite exponential moments are required; a forthcoming work will justify this with counterexamples.
	
	The almost sure convergence
	\[
	\lim_{n \to \infty} \frac{1}{n} \log \norm{A^n v} = L_1(\mu) \quad \Pp_\mu\text{-a.s.}
	\]
	follows from a result by Furstenberg and Kifer~\cite{FK83}. Extending it to our setting involves subtle issues due to singularities in the projective action (see Section~\ref{FK}).
	
	Finally, obtaining an explicit H\"older modulus is technically demanding, contributing to the length of this work.

\bigskip


The paper is organized as follows. Section~\ref{statements} presents the main results. Section~\ref{Wasserstein} introduces the Wasserstein metric and key tools for measures and H\"older functions. Section~\ref{Setting} establishes core inequalities used throughout. In Section~\ref{FK}, we extend Furstenberg-Kifer theory to non-compact, non-invertible settings under exponential moment assumptions. Section~\ref{Markov} proves uniform ergodicity of the associated Markov operator. Section~\ref{Holder} establishes H\"older continuity of Lyapunov exponents. Section~\ref{Large} derives stable exponential large deviation estimates. Finally, Section~\ref{Schro} applies the theory to Schr\"odinger cocycles.

\bigskip

\section{Statement of results}\label{statements}
Let $\Scal=\Mat_m(\R)$ be the (locally compact) space of square $m\times m$ matrices.
Given  $0<p\leq 1$, consider the (exponential) moments of order $p$  of a measure $\mu\in \Prob(\Scal)$,
\begin{equation}
	\label{Theta p}
	{\bar \Theta}_p(\mu):= \int \norm{g}^p\, d\mu(g) ,
\end{equation}
and
\begin{equation}
	\label{under Theta p}
	{\underline\Theta}_p(\mu):= \sup_{\norm{v}=1} \int \norm{g\, v}^{-p}\, d\mu(g) .
\end{equation}
The space
$$\Prob_p(\Scal):=\left\{\mu\in \Prob(\Scal) \, \colon\,    \int \norm{g}^p \ d\mu(g)<+\infty\right\} $$
was extensively studied by C. Villani in the context of Optimal Transport, see~\cite{CV}. This space admits a natural metric known as the \textit{Wasserstein distance}, whose main properties are explained in Section~\ref{Wasserstein}.
 
A coupling between two measures $\mu, \nu\in \Prob(\Scal)$ is any probability $\pi\in \Prob(\Scal \times \Scal)$ with marginal measures $\mu$ and $\nu$. We denote by  $\Pi(\mu,\nu)$ the set of all couplings between $\mu$ and $\nu$.
The Wasserstein metric of order $p$ is 
$W_p: \Prob_p(\Scal)\times \Prob_p(\Scal)\to [0, +\infty]$, defined by
$$W_p(\mu,\nu):=\inf_{\pi\in \Pi(\mu,\nu)} \,  \iint_{\Scal\times \Scal} \norm{g-g'}^p \ d\pi(g,g') . $$

Given $C\geq1$ and $0<p \leq 1$ we consider  the space of probability measures with finite exponential  moments
\begin{align*}
	\Mcal:=\left\{\mu\in  \Prob(\Scal) \colon\,    \bar \Theta_p(\mu) \leq C \, \text{ and }\, \underline \Theta_p(\mu) \leq C\; \right\}. 
\end{align*} 
By Proposition~\ref{MCp completeness}, this space is complete with respect to the Wasserstein distance $W_p$.

%

\begin{definition}
	\label{def quasi-irreducibility}
	We say that a measure $\mu\in \Mcal$ is quasi-irreducible if for any proper subspace $E\subset \R^m$ such that $g\, E\subseteq E$ for $\mu$-almost every $g\in \Mat_m(\R)$, the  cocycle  induced on the subspace $E$ has first Lyapunov exponent equal to $L_1(\mu)$.
\end{definition}

\begin{definition}
	\label{def mu stationary}
	We say that a measure $\eta\in \Prob(\Pp(\R^m))$ is $\mu$-stationary if $\eta(\{ \hat v\in \Prob(\Pp(\R^m)) \colon g\, v=0 \})=0$ for $\mu$ -almost every $g$ and for every Borel set $B\subset \Pp(\R^m)$, writing
	$$g^{-1}(B):=\{ \hat v\in\Pp(\R^m))\colon g\, v\neq 0 \, \text{ and }\, \hat g\, \hat v\in B \} $$
	we have
	$$ \eta(B)=\int \eta(g^{-1}(B))\, d\mu(g) $$
\end{definition}

\begin{theorem}
	\label{thm non invertible Holder LE }
	Given $\mu\in \Mcal$ such that
	\begin{enumerate}
		\item $\mu$ is quasi-irreducible;
		\item $L_1(\mu)>L_2(\mu)$;
	\end{enumerate}
	 the function $\Mcal\ni \nu \mapsto L_1(\nu)$ is H\"older continuous in a small neighborhood of $\mu$ with respect to $W_p$.
\end{theorem}

\begin{remark}
	The H\"older exponent in the regularity of the Lyapunov exponent depends on $p$, $C$ and $\mu\in \Mcal$.
\end{remark}

In order to extend this result to other Lyapunov exponents, we introduce the array of measurements, for $1\leq k\leq m$,
$${\underline\Theta}_p^k(\mu):= \sup_{\norm{v\wedge \cdots \wedge v_k}=1} \int \norm{\wedge_k g\, (v_1\wedge \cdots \wedge v_k)}^{-p}\, d\mu(g) ,$$
and define
$$ \Mcal(k):=\{\mu\in \Prob(\Mat_m(\R))\, \colon\, \bar \Theta_p(\mu)\leq C\, \text{  and }\, \underline \Theta_p^k(\mu)\leq C\, \}.$$

The following straightforward result, whose proof we omit, characterizes these moments in terms of the Grassmannian $\Gr_k(\R^m)$  of $k$-dimensional linear subspaces of $\R^m$.

\begin{proposition} 
	For any measure $\mu\in \Prob(\Mat_m(\R)$,
	$$ \underline{\Theta}_p^k(\mu) =  \sup_{F\in \Gr_k(\R^m)} \int \frac{1}{|\det(g_{\vert F})|^p }\, d\mu(g) , $$	
	where  $\det(g_{\vert F})$ stands for  the volume dilation factor  of
	the linear map  $g_{\vert F}:F\to g(F)$, i.e.,  the product of its singular values.
\end{proposition}

In Section~\ref{Setting}, we will show that the bounds on $\bar{\Theta}_p(\mu)$ and $\underline{\Theta}_{p'}^k(\mu)$ -- which are guaranteed for $\mu \in \Mcal(k)$ -- imply corresponding bounds on all intermediate moments $\underline{\Theta}_{p'}^j(\mu)$ for $1 \leq j < k$, possibly for smaller exponents $0 < p' < p$, see Lemma~\ref{bounds on Theta j}.

\begin{definition}
	\label{def k quasi-irreducibility}
	We say that a measure $\mu\in \Mcal(k)$ is $k$-quasi-irreducible if for all $1\leq j\leq k$ the measures $\wedge_j\mu=\int \delta_{\wedge_j g}\, d\mu(g)$ are quasi-irreducible.
\end{definition}

\begin{remark}
If $\mu$ generates a Zariski-dense sub-semigroup then
irreducibility of $\wedge_j\mu$, for all $1\leq j\leq k$,
is automatic. In particular, $\mu$ is
$k$-quasi-irreducible.
\end{remark}

\begin{corollary}
	\label{non-invertible FK coro}
	Given $1\leq k\leq m$, $\mu\in \Mcal(k)$ such that
	\begin{enumerate}
		\item $\mu$ is $k$-quasi-irreducible;
		\item $L_1(\mu)>\cdots >L_k(\mu)$;
	\end{enumerate}
	the function $\Mcal(k)\ni \nu \mapsto (L_1(\nu) \ldots, L_k(\nu))$ is H\"older continuous in a small neighborhood of $\mu$.
\end{corollary}

\begin{corollary}
\label{coro main them det moment}
 $\Mcal(m)$ consists of  $\mu\in \Prob_p(\Scal)$ such that
$$ \bar \Theta_p(\mu)\leq C\; \text{ and }\; \underline \Theta_p^m(\mu)= \int \frac{1}{|\det (g)|^p}\, d\mu(g) \leq C .$$
In this case, under the same simplicity and $m$-quasi irreducibility hypothesis, all Lyapunov exponents are finite and locally H\"older.
\end{corollary}

 Let $\Sigma$ be a compact metric space and $\mu$ a probability measure on $\Sigma$.
 Consider the space of sequences $\Omega:=\Sigma^\Z$ endowed
 with the Bernoulli measure $\Pp_\mu:=\mu^\Z$.
 A function $A:\Omega\to \SL_m(\R)$ is said to be \textit{locally constant}  if $A(\omega)=A(\omega_0)$, i.e., it depends only on the $0$-th coordinate  of the argument.
 Given $p>0$, we say that $A$ has a $p$-exponential moment if
\begin{equation}
\label{Theta p (A)}
 \Theta_p(A):= \int_\omega \norm{A(\omega)}^p\, d\Pp_\mu(\omega)<+\infty .
\end{equation}
Let $\mathscr{C}_p(\Omega, \SL_m(\R))$ denote 
 the space of locally constant functions with finite $p$-exponential moment.
 A function $A\in \mathscr{C}_p(\Omega, \SL_m(\R))$ determines a linear cocycle $F_A:\Omega\times\R^m\to \Omega\times\R^m$,
 $F_A(\omega, v):=(\sigma\omega, A(\omega)\, v)$,
 with finite moment. This holds because
 $\log \norm{A(\omega)}\lesssim \norm{A(\omega)}^p$.
 The cocycle $F_A$ can be naturally identified with the cocycle determined by the probability measure 
 $\nu:= A_\ast\Pp_\mu\in \Prob(\SL_m(\R))$. In fact
 we have the following commutative diagram 
 $$\begin{CD}
 \Omega \times\R^m  @>F_A>> \Omega \times\R^m \\
 @V h VV @VV h V\\
 X \times\R^m @>F >> X \times\R^m
 \end{CD}$$
 where $X=\SL_m(\R)^\N$, $F(\omega, v)=(\sigma\omega, \omega_0\, v)$, and $h:\Omega \times\R^m \to X \times\R^m$ is the map $h(\omega, v):=\left( (A(\sigma^n\omega))_{n\in \Z}, v\right) $.

We endow the space $\mathscr{C}_p(\Omega, \SL_m(\R))$ with the $L^1$-distance
$$ d(A,B):=\int_\Omega \norm{A-B}\, d\Pp_\mu .$$ 
 \begin{corollary}
 	\label{thm L1 Holder LE}
 	Given $A\in \mathscr{C}_p(\Omega, \SL_m(\R))$ such that
 	\begin{enumerate}
 		\item $A$ is quasi-irreducible;
 		\item $L_1(A)>L_2(A)$;
 	\end{enumerate}
 	the function $\mathscr{C}_p(\Omega, \SL_m(\R))\ni B \mapsto L_1(B)$ is H\"older continuous in a small neighborhood of $A$ w.r.t. the $L^1$-distance above.
 \end{corollary}

 This result was proved by  P. Barrientos and D. Malicet in a recent article~\cite[Proposition X]{BarrientosMalicet2025}.

 Under the same hypothesis, stable large deviation estimates hold.
 
 \begin{theorem}
 	\label{thm large deviation estimates}
 	Given $\mu\in \Mcal$ such that
 	\begin{enumerate}
 		\item $\mu$ is quasi-irreducible;
 		\item $L_1(\mu)>L_2(\mu)$;
 	\end{enumerate}
 	there are positive constants $c$, $C$ and $\delta$ such that for every measure $\nu\in \Mcal$ with $W_p(\nu, \mu)<\delta$, all $n\geq 1$ and every $\varepsilon>0$ and $\hat v\in \Pp(\R^m)$,
 	$$ \nu^\N\left\{\, \omega\in \Mat_m(\R)^\N\colon \, \left\vert
 	\frac{1}{n}\, \log \norm{A^n(\omega)\, v} - L_1(\nu)  \right\vert >\varepsilon  \, \right\}<C\, e^{-c(\varepsilon)\, n} ,$$
 	with $c(\varepsilon)\sim \varepsilon^2/(\log(1/\varepsilon))$.
 \end{theorem}

\section{Spaces of Measures}\label{Wasserstein}
Suppose that $X$ is a Polish space, i.e., a complete and separable metric space and consider the set $\Prob(X)$ of Borel probability measures on $X$.

\begin{definition}
	We say that a sequence $\mu_n\in\Prob(X)$ converges weakly to $\mu$, and write  $\mu_n\overset{\text{weak}}{\longrightarrow}\mu$, if $$\int\varphi \ d\mu_n\longrightarrow\int\varphi \ d\mu,\quad \forall\varphi\in C_b(X), $$ 
	where $C_b(X)$ denotes the space of bounded and continuous functions on $X$. This defines a Hausdorff topology  on $\Prob(X)$ called the weak topology. Note that $\mu=\nu$ if and only if $\int\varphi \ d\mu=\int \varphi \ d\nu$ $\forall\varphi\in C_b(X)$.
\end{definition}

If $X$ is a locally compact Polish space,
let $C_0(X)$ denote the space of continuous functions going to $0$ at infinity.
In this case, via Riesz's Theorem, we can identify
the space of  finite signed measures  $\Mea(X)$, normed by the total variation, with the dual of $C_0(X)$.
We can then introduce the weak $\ast$ topology on $\Mea(X)$ and in particular in $\Prob(X)$.

\begin{definition}
	We say that $\mu_n\overset{\text{weak $\ast$}}{\longrightarrow}\mu$ if $$\int\varphi \ d\mu_n\longrightarrow\int\varphi \ d\mu,\quad \forall\varphi\in C_0(X) .$$ 
	This defines the weak $\ast$ topology on $\Prob(X)$.
\end{definition}

\begin{remark}
	When $X$ is a locally compact Polish space, weak and weak $\ast$ convergence coincide in $\Prob(X)$, see \cite{CV2}.
\end{remark}

Let $(X,d)$ be a Polish space. 
Given $p>0$   define   $$\Prob_p(X):=\left\{\mu\in \Prob(X)\, \colon\,  \exists x_0 \in X : \int_X d(x_0,x)^p \ d\mu(x)<+\infty\right\}.$$
The Wasserstein distance of order $p$ is the metric $$W_p: \Prob_p(X)\times \Prob_p(X)\to [0,+\infty] ,$$ defined as
$$W_p(\mu,\nu):=\inf_{\pi\in\Pi(\mu,\nu)}\left(\int\int_{X\times X} d(x,y)^p \ d\pi(x,y)\right)  \text{ for }   0<p<1$$ or
$$W_p(\mu,\nu):=\inf_{\pi\in\Pi(\mu,\nu)}\left(\int\int_{X\times X} d(x,y)^p \ d\pi(x,y)\right)^{1/p} \ \text{ when } \ p\geq 1 .$$ 
Here $\Pi(\mu,\nu)$ denotes the set of all the couplings between $\mu$ and $\nu$, i.e.,   probability measures on $X\times X$ with marginals $\mu$ and $\nu$, mentioned at the introduction.

The Wasserstein distance $W_1$ can also be defined, by the Kantorovich-Rubinstein duality, as 
$$W_1(\mu,\nu)=\sup_{\Lip(\varphi)\leq 1}\int_X \varphi \, d(\mu-\nu),$$ where $\Lip(\varphi)$ is the Lipschitz constant of $\varphi$ and $\int_X \varphi \, d(\mu-\nu):=\int_X \varphi \, d\mu - \int_X \varphi \, d\nu$.

\begin{remark}
	\label{remark Wp 0<p<1}
	When $0<p<1$, notice that $d^p(x,y):=d(x,y)^p$ is a distance on $X$ that is topologically equivalent to $d$. In particular $\Prob_p(X,d)$ and $\Prob_1(X,d^p)$ coincide as metric spaces, with the distances $W_p$ and $W_1$ respectively. Thus, we have 
	$$W_p(\mu,\nu) =\sup_{\varphi\in\Hol(X)}\int_X \varphi \ d(\mu-\nu),$$ where $\Hol(X)$ is the space of $p$-H\"older continuous functions $\varphi:X\to\R$.
\end{remark}

Next theorem, see \cite[Theorem 7.12]{CV2},  characterizes the convergence in the Wasserstein topology.

\begin{theorem}
	\label{thm: Wasserstein convergence}
	Given $0<p<\infty$, a sequence  $(\mu_k)_{k\in\mathbb{N}}$ and a  measure $\mu$ in $\Prob_p(X)$, the following statements are equivalent:
	\begin{enumerate}
		\item[(i)] $\lim_{k\to \infty} W_p(\mu_k,\mu)= 0$;
		\item[(ii)] $\lim_{k\to \infty} \mu_k=\mu$  weakly, and  for some  $x_0\in X$ (and thus any),
		$$\lim_{R\rightarrow\infty}\limsup_{k\rightarrow\infty}\int_{d(x_0,x)\geq R}d(x_0,x)^p \, d\mu_k(x)=0;$$
		\item[(iii)]  $\lim_{k\to \infty} \mu_k=\mu$  weakly, and for some $x_0\in X$ (and thus any), 
		$$\int d(x_0,x)^p \, d\mu_k(x)\longrightarrow\int d(x_0,x)^p \, d\mu(x)  \text{ as }   k\to\infty;$$
		\item[(iv)] For any   $\varphi\in C(X)$  satisfying the growth condition: $$\exists x_0\in X\; \exists C<\infty \; \forall x\in X \quad \left|\varphi(x)\right|\leq C(1+d(x_0,x)^p) ,$$ 
		$$\lim_{k\to \infty} \int \varphi \, d\mu_k = \int \varphi \, d\mu  .$$
		
	\end{enumerate}

\end{theorem}

\begin{theorem}
	Given $p>0$ and a complete and separable metric space $(X,d)$,   then $(\Prob_p(X),W_p)$ is a complete and separable metric space.
\end{theorem}

\begin{proof}
	The case where $p\geq 1$ is proved, see \cite[Theorem 6.18]{CV}. When $0<p<1$, the conclusion reduces, by Remark~\ref{remark Wp 0<p<1}, to the case $p=1$.
\end{proof}

\begin{proposition}
	\label{MCp completeness}
	Given $0<p\leq 1$ and $C<\infty$, \, $\Mcal$ is complete.
\end{proposition}

\begin{proof}
	Consider the metric space $\Mat_m(\R)$ equipped with the distance $d(A,B)=\norm{A-B}$, and let $\{\mu_n\}_{n\in\N}$ be a sequence of measures in $\Mcal$ such that $\mu_n \overset{W_p}{\longrightarrow}\mu$. 
	For every $n\in\N$, the following inequalities hold:
	$$\int_{\Mat_m(\R)} \norm{g}^p\, d\mu_n(g) \leq C\quad \text{ and } \quad \sup_{\norm{v}=1} \int \frac{1}{\norm{g v}^{p}} \, d\mu_n(g) \leq C.$$
	
	We aim to prove that $\mu\in \Mcal$.
	
	First, note that the function $\varphi(g):=\norm{g}^p$  is $p$-H\"older continuous. Taking the limit as $n\to \infty$ in the first inequality, the convergence of $\mu_n$ to $\mu$ in $W_p$ implies:
	$$\int\norm{g}^p d\mu(g)\leq C.$$ 
	
	Next, consider any unit vector  representative $v$ of a point $\hat v\in \Pp(\R^m)$. For $T>0$, define the Lipschitz function
	$$\varphi_{T,v}(g):=\min\bigg\{T,\frac{1}{\norm{gv}^p}\bigg\}.$$ 
	
	Since
	$$\int \varphi_{T,v}(g)d\mu_n(g)\leq\int\frac{1}{\norm{gv}^p} d\mu_n(g)\leq C ,$$ 
	taking the limit as $n\to\infty$ and using the $W_p$-convergence of $\mu_n$ to $\mu$, we obtain:
	$$\int\varphi_{T,v}(g) d\mu(g)\leq C.$$
	
	By the monotone convergence theorem, letting $T\to \infty$, we deduce:
	$$\int\frac{1}{\norm{gv}^p} d\mu(g)=\lim_{T\rightarrow+\infty}\int\varphi_{T,v}(g) d\mu(g)\leq C.$$
	
	Finally, taking the supremum over all unit vectors $v$, we conclude:
	$$\sup_{\norm{v}=1}\int\frac{1}{\norm{gv}^p} d\mu(g)\leq C,$$
	
	This establishes that $\mu\in\Mcal$, as desired.
\end{proof}

\begin{remark}
	The boundedness condition within  $\Mcal$ is required to guarantee the modulus of continuity to be established in the subsequent sections.
	
\end{remark}

The following result is straightforward, therefore its proof is omitted.

\begin{proposition}
	$\mathcal{M}^p_C$ is convex.
\end{proposition}

Consider the convex space
\begin{align*}
	\mathcal{M}^p:=& \left\{\mu\in  \Prob(\Mat_m(\R)) \colon\,    \int_{\Mat_m(\R)} \norm{g}^p\, d\mu(g) <\infty  , \right. \\
	&\hspace{4cm} \left.  \sup_{\norm{v}=1} \int_{\Mat_m(\R)} \frac{1}{\norm{g v}^{p}} \, d\mu(g) <\infty \; \right\}. 
\end{align*} 
endowed with the Wasserstein metric $W_p$,
where $0<p\leq 1$.

\begin{proposition}\label{cont1}
	The convolution is locally Lipschitz in $\mathcal{M}^p$ with res\-pect to the Wasserstein distance $W_p$.
\end{proposition}
\begin{proof}
	First, note that the mapping  $\Mat_m(\R)\ni g\mapsto \norm{g}^p$ is $p$-H\"older continuous.
	As a result, the function
	$\mathcal{M}^p\ni \nu \mapsto \bar \Theta_p(\nu):=\int\norm{g}^p \ d\nu(g)$
	is continuous with respect to the $W_p$ metric.
	Given $\nu\in \mathcal{M}^p$ and $\varphi\in \Hol(\Mat_m(\R))$ with H\"older constant $\leq 1$, consider the function
	$$g\mapsto \int \varphi(g g')d\nu(g') .$$ This function is  $p$-H\"older with constant $\bar \Theta_p(\nu)$. To verify this, observe:
	\begin{align*}
		&\left| \int \varphi(g g') \, d\nu(g') - \int \varphi(g_1 g') \, d\nu(g') \right| 
		\leq \int \left| \varphi(g g')-\varphi(g_1 g') \right| \ d\nu(g')\\
		&\hspace{3cm} \leq   \int\norm{g g' - g_1 g'}^p \ d\nu(g')   \leq \bar \Theta_p(\nu)\,\norm{g-g_1}^p.
	\end{align*}
	
	Now, consider $\mu, \mu', \nu\in \mathcal{M}^p$. To estimate $W_p(\mu\ast \nu, \mu'\ast\nu)$
	take $\varphi\in \Hol(\Mat_m(\R))$ with H\"older constant $\leq 1$.
	Using the previous result, we have:
	\begin{align*}
		&\left|\iint\varphi(g g')d\mu(g)d\nu(g')-\iint\varphi(g g')d\mu'(g)d\nu(g')\right| \leq \bar \Theta_p(\nu)\, W_p(\mu,\mu').
	\end{align*}
	Taking the supremum over all such $\varphi$,
	it follows that:
	$$W_p(\mu\ast \nu,\mu'\ast\nu)\leq \bar \Theta_p(\nu)\, W_p(\mu,\mu').$$ 
	Similarly, one can show that:
	$$W_p(\mu\ast \nu,\mu\ast\nu')\leq \bar \Theta_p(\mu)\, W_p(\nu,\nu').$$ 
	Combining these results, we obtain:
	\begin{align*}
		W_p(\mu\ast \nu, \mu'\ast \nu') & \leq W_p(\mu\ast \nu, \mu'\ast \nu)+W_p(\mu'\ast \nu, \mu'\ast \nu')\\
		& \leq \bar \Theta_p(\nu)\,  W_p(\mu, \mu')+\bar \Theta_p(\mu')\, W_p(\nu, \nu')\\
		&\leq \max\left\{\bar \Theta_p(\nu), \bar \Theta_p(\mu') \right\}\, \left( W_p(\mu, \mu')+ W_p(\nu, \nu') \right) .
	\end{align*}
	Since  $\bar \Theta_p$  is continuous and locally bounded, it follows that the convolution of measures in $\mathcal{M}^p$ is locally Lipschitz with respect to $W_p$.
\end{proof}

Let $(X,d)$ be a complete separable metric.
Given $\mu\in \Prob_p(X,d)$ and $x_0\in X$,
consider the moment
$$ \Theta_p(\mu):=\int_X d(x_0,x)^p\, d\mu(x). $$
Given a measurable subset $F\subset X$, we denote by $\mu_F$  the restriction of $\mu$ to $F$, i.e., $\mu_F(A):=\mu(A\cap F)$
for any Borel set $A\subseteq X$.
\begin{lemma}
	\label{lema W(mu bar mu)}
Given an open set $\Bscr\subset X$ with $0<\mu(\Bscr)\ll 1$, the probability measure $\bar \mu=(1-\mu(\Bscr))^{-1}\, \mu_{\Bscr^\complement}$ satisfies
$$ W_p(\mu, \bar \mu) \leq 2\, \sqrt{\mu(\Bscr)} \, \Theta_{2p}(\mu)^{1\over 2} .$$
\end{lemma}

\begin{proof}
Consider the measure on $X\times X$,
$$ \pi=\int_{\Bscr^\complement}\delta_{(x,x)}\, d\mu(x)
+ (1-\mu(\Bscr))^{-1}\, (\mu_{\Bscr}\times \mu_{\Bscr^\complement}). $$
The first term has mass $1-\mu(\Bscr)$, while the second has mass $\mu(B)$. Hence $\pi$ is a probability measure.
The marginal measures of $\pi$ are respectively 
$\mu_{\Bscr^\complement}+\mu_{\Bscr}=\mu$ and
$\mu_{\Bscr^\complement} +\frac{\mu(\Bscr)}{1-\mu(\Bscr)}\, \mu_{\Bscr^\complement} =\bar \mu$, which shows that $\pi$ is a coupling between $\mu$ and $\bar \mu$. Therefore, using Cauchy-Schwarz
\begin{align*}
W_p(\mu, \bar \mu) &\leq\iint d(x,y)^p\, d\pi(x,y) 
=\frac{1}{1-\mu(\Bscr)}\, \int_{\Bscr}\int_{\Bscr^\complement} d(x,y)^p\,d\mu(x)\, d\mu(y) \\
&\leq 	\frac{1}{1-\mu(\Bscr)}\, \int_{\Bscr}\int_{X} [\, \, d(x_0,x)+d(x_0,y)\, ]^p\,d\mu(x)\, d\mu(y) \\
&\leq 	\frac{1}{1-\mu(\Bscr)}\, \int_{\Bscr}\int_{X}  d(x_0,x)^p+d(x_0,y)^p\,d\mu(x)\, d\mu(y) \\
&\leq \frac{1}{1-\mu(\Bscr)}\,\left[ \, \int_{\Bscr} d(x_0,x)^p\, d\mu(x) + \mu(\Bscr)\, \Theta_p(\mu) \,\right]\\
&\leq \frac{1}{1-\mu(\Bscr)}\,\left[ \, \sqrt{\mu(\Bscr)} \, \sqrt{\Theta_{2p}(\mu))}+ \mu(\Bscr)\, \Theta_p(\mu) \,\right]\\
&\leq 2\, \sqrt{\mu(\Bscr)} \, \Theta_{2p}(\mu)^{1\over 2} ,
\end{align*}
the last inequality because
\( \mu(\Bscr) \ll \sqrt{\mu(\Bscr)} \) and
\( \Theta_p(\mu) \leq \Theta_{2p}(\mu)^{1\over 2} \).
\end{proof}

For measurable functions $\psi:X\to\R$ which are Lipschitz outside  a small set $\Bscr\subset X$, the following two lemmas provide  explicit moduli of continuity for the map $\Prob_p(X)\ni \mu \mapsto \int \psi\, d\mu$.

\begin{lemma}
\label{lemma almost Holder continuity I}
Given constants $1\leq L<\infty$, $0<p\leq \frac{1}{2}$, an open set $\Bscr\subset X$  and a measurable function $\psi:X\to\R$
such that
$$   \abs{\psi(x)-\psi(y)}\leq L\, d(x,y)^p\quad \forall\, x,y\in X\setminus \Bscr ,$$
then for any probability measures $\mu, \nu\in \Prob_p(X,d)$,
	
\begin{align*}
		\left|\int\psi\, d\mu -\int \psi\, d\nu \right|
		\leq L  \left\{ \; W_p(\mu,\nu) \right.  &+\,  4\,\sqrt{\mu(\Bscr)}\, \left( \norm{\psi}_{L^2(\mu)} + \Theta_{2p}(\mu)^{1\over 2}  \right) \\
		&   + \, 4\,\sqrt{\nu(\Bscr)}\, \left( \norm{\psi}_{L^2(\nu)} +\Theta_{2p}(\nu)^{1\over 2}  \right)\left.\phantom{W_p} \right\}.
\end{align*}
\end{lemma}

\begin{proof}
In the inequality above, we refer to the term \( L\, W_p(\mu, \nu) \) as the \emph{leading term}. The constant $C$ in the remainder expression
\[
C\, L\, \left( \Xi_p(\mu, \psi) + \Xi_p(\nu, \psi) \right),
\]
where
\[
\Xi_p(\mu, \psi) := \sqrt{\mu(\Bscr)} \left( \|\psi\|_{L^2(\mu)} + \Theta_{2p}(\mu)^{1/2} \right),
\]
will be called the  \emph{remainder coefficient}.

In the proof, we make use of the probability measures \( \bar\mu \) and \( \bar\nu \) supported on \( \Bscr^\complement \), as introduced in Lemma~\ref{lema W(mu bar mu)}.

Since \( \Bscr \) is open, its complement \( \Bscr^\complement = X \setminus \Bscr \) is closed, and hence a complete metric space.

Applying the Cauchy--Schwarz inequality, we obtain:
\begin{align*}
	\left|\int\psi\, d\mu -\int \psi\, d\nu \right| \hspace{-2cm} &\hspace{2cm}\leq
	\left|\int_{\Bscr^\complement} \psi\, d \mu -\int_{\Bscr^\complement} \psi \, d \nu \right| + \left| \int_{\Bscr}\psi\, d\mu \right| +  \left| \int_{\Bscr}\psi\, d\nu \right| \\
	&\leq
	\left|\int_{\Bscr^\complement} \psi\, d \mu -\int_{\Bscr^\complement} \psi \, d \nu \right| + 
	\sqrt{\mu(\Bscr)}\,   \norm{\psi}_{L^2(\mu)} +
	\sqrt{\nu(\Bscr)}\,   \norm{\psi}_{L^2(\nu)} ,
\end{align*}
where the last two terms contribute at most a constant of size 1 each to the remainder coefficient \( C \).

Next, we turn our attention to the new leading term:
\begin{align*}
	\left|\int_{\Bscr^\complement} \psi\, d \mu \, -\int_{\Bscr^\complement} \psi \, d \nu \right| \hspace{-3cm}&\hspace{3cm} \leq
	\left|\int \psi\, d \bar \mu -\int \psi \, d \bar \nu \right|   + \left| \int_{\Bscr^\complement}\psi\, d\mu -\int \psi\, d\bar \mu \right|  \\
	&\hspace{4cm}+ 
	\left|\int_{\Bscr^\complement} \psi\, d \nu -\int \psi \, d \bar \nu \right|  \\
	&\leq L\, W_p(\bar \mu, \bar \nu) + \left(\frac{1}{1-\mu(\Bscr)}-1\right) \,  \smallint |\psi|\, d\mu  + \left(\frac{1}{1-\nu(\Bscr)}-1\right) \,  \smallint |\psi|\, d\nu\\
	&\leq L\, W_p(\bar \mu, \bar \nu) + \frac{\mu(\Bscr)}{1-\mu(\Bscr)}\, \norm{\psi}_{L^1(\mu)} +
	\frac{\nu(\Bscr)}{1-\nu(\Bscr)}\, \norm{\psi}_{L^1(\nu)} .
\end{align*}

Since \( \mu(\Bscr) \leq \sqrt{\mu(\Bscr)} \) and
\( \|\psi\|_{L^1(\mu)} \leq \|\psi\|_{L^2(\mu)} \),
these terms contribute at most a small additive constant to the remainder coefficient.

It remains to estimate the final leading term. By the triangle inequality and Lemma~\ref{lema W(mu bar mu)}, we have:
\begin{align*}
	L\, W_p(\bar\mu, \bar\nu)
	&\leq L\, W_p(\mu, \nu)
	+ L\, W_p(\mu, \bar\mu)
	+ L\, W_p(\nu, \bar\nu) \\
	&\leq L\, W_p(\mu, \nu)
	+ 2L\, \sqrt{\mu(\Bscr)}\, \Theta_{2p}(\mu)^{1/2}
	+ 2L\, \sqrt{\nu(\Bscr)}\, \Theta_{2p}(\nu)^{1/2},
\end{align*}
where the final two terms contribute an additional constant of size 2 to the remainder coefficient \( C \).

Combining all these bounds, we conclude the proof with a total remainder coefficient given by \( C = 4 \).
\end{proof}

In the following, closely related lemma, we assume that $X$ is a finite-dimensional Euclidean space. This result will be applied with $X = \Mat_m(\R)$. Notice the assumptions are stronger, but the conclusion is also sharper when $T \ll L$.

\begin{lemma}
\label{lemma almost Holder continuity II}
Consider constants $L, T <\infty$, $0<p\leq 1$ and
a measurable function $\psi:X\to  \R\cup \{-\infty\}$ which is continuous at all points $x\in X$ where $\psi(x)> -\infty$. Define $\Bscr:=\{ x\in X\colon |\psi(x)|>T \}$. \, If
$$   \abs{\psi(x)-\psi(y)}\leq L\, d(x,y)^p\quad \forall\, x,y\in X\setminus \Bscr ,$$	
then for any probability measures $\mu, \nu\in \Prob_p(X,d)$,

\begin{align*}
	\left|\int\psi\, d\mu -\int \psi\, d\nu \right|
	\leq L  \, W_p(\mu,\nu)   &   +   \sqrt{\mu(\Bscr)}\,  \norm{\psi}_{L^2(\mu)}  +  \sqrt{\nu(\Bscr)}\, \norm{\psi}_{L^2(\nu)}  \\
	&+\, T\, \left(\mu(\Bscr)+\nu(\Bscr) \right) .
\end{align*}
\end{lemma}

\begin{proof}
Define the truncated function $\bar\psi : X \to \R$ by
\[
\bar\psi(x) := 
\begin{cases}
	\phantom{+}T & \text{if } \psi(x) > T, \\
	\psi(x) & \text{if } -T \leq \psi(x) \leq T, \\
	-T & \text{if } \psi(x) < -T,
\end{cases}
\]
and note that $\bar\psi \in \Hol(X)$ with $v_p(\bar\psi) \leq L$.

Indeed, given $x, y \in X$ such that $-T \leq \psi(x), \psi(y) \leq T$, we have $x, y \in X \setminus \Bscr$, and by assumption,
\[
|\bar\psi(x) - \bar\psi(y)| = |\psi(x) - \psi(y)| \leq L\, d(x,y)^p.
\]

If instead $\psi(x) < -T < T < \psi(y)$, then by the intermediate value theorem, there exist points $x_*$ and $y_*$ on the line segment joining $x$ and $y$ such that $\psi(x_*) = -T$ and $\psi(y_*) = T$. Hence $x_*, y_* \in X \setminus \Bscr$ and
\[
|\bar\psi(x) - \bar\psi(y)| \leq |\psi(x_*) - \psi(y_*)| \leq L\, d(x_*, y_*)^p \leq L\, d(x,y)^p.
\]
All remaining cases are treated in a similar manner.

Therefore,
\begin{align*}
	\left|\int_{\Bscr^\complement} \psi\, d\mu - \int_{\Bscr^\complement} \psi\, d\nu \right|
	&\leq \left| \int \bar\psi\, d\mu - \int \bar\psi\, d\nu \right|
	+ \int \left| \mathbf{1}_{\Bscr^\complement}\, \psi - \bar\psi \right|\, d(\mu + \nu) \\
	&\leq L\, W_p(\mu, \nu) + \int_\Bscr |\bar\psi|\, d(\mu + \nu) \\
	&\leq L\, W_p(\mu, \nu) + T\left( \mu(\Bscr) + \nu(\Bscr) \right).
\end{align*}

The conclusion follows by combining this with the first inequality in the proof of Lemma~\ref{lemma almost Holder continuity I}.
\end{proof}

\section{Some Inequalities}
\label{Setting}
In this section we have gathered several inequalities that will be used in the rest of the paper.

\begin{proposition}
\label{Theta_n(mu n)}
Given $\mu\in \Prob(\Mat_m(\R))$ and $n\in\N$,
$$   \bar{\Theta}_p(\mu^n)\leq \bar{\Theta}_p(\mu)^n \quad \text{ and } \quad  \underline{\Theta}_p(\mu^n)\leq \underline{\Theta}_p(\mu)^n.$$
\end{proposition}

\begin{proof}
We have
\begin{align*}
 \bar{\Theta}_p(\mu^n)&=\int\!\cdots\!\int\norm {g_n\cdots g_1}^p\, d\mu(g_n)\, \cdots\, d\mu(g_1)\\
 &\leq \int\!\cdots\!\int\norm {g_n}^p\,\cdots\, \norm{g_1}^p\, d\mu(g_n)\, \cdots\, d\mu(g_1)\\
 &=\prod_{i=1}^n \int\norm{g_i}^p\, d\mu(g_i) =\bar \Theta_p(\mu)^n .
\end{align*}
Similarly
\begin{align*}
	\underline{\Theta}_p(\mu^n)&=\sup_{\hat v} \int\!\cdots\!\int\frac{1}{\norm {g_n\cdots g_1\, v}^p}\, d\mu(g_n)\, \cdots\, d\mu(g_1)\\
	&= \sup_{\hat v} \int\!\cdots\!\int \prod_{i=1}^n \frac{1}{\norm {g_i\, \frac{g_{i-1}\cdots g_1\, v}{\norm{g_{i-1}\cdots g_1\, v}} }^p}\, d\mu(g_n)\, \cdots\, d\mu(g_1)\leq \underline{\Theta}_p(\mu)^n .
\end{align*}
The last inequality is derived integrating from left to right, that is from $g_n$ to $g_1$, making use of the estimate
$\int \norm{g_i\, v_i}^{-p}\, d\mu(g_i)\leq \underline{\Theta}_p(\mu)$  with
$v_i= \frac{g_{i-1}\cdots g_1\, v}{\norm{g_{i-1}\cdots g_1\, v}}$.
\end{proof}

Given $T>0$ and $\hat v\in \Pp(\R^m)$, consider the  set 
\begin{equation}
\label{BT(v) defin}
\Bscr_T(\hat v):= \left\{ g\in \Mat_m(\R) \,  \colon \, \log \norm{g\, v}<-T \, \text{ or }\, \log \norm{g}>T \right\} .
\end{equation}

\begin{lemma}
	\label{bound BT(v)}
	Given $\mu\in \Mcal$ and $\hat v\in \Pp(\R^m)$,
	$\mu \left( \Bscr_T(\hat v) \right) \, \leq 2\,C\, e^{-p\, T} . $
\end{lemma}

\begin{proof}
	The set $\Bscr_T(\hat v)$ is the union of two subsets
	\begin{align*}
		\Bscr^+_T &:=\{ g\in \Mat_m(\R) \colon \log \norm{g}  > T\;  \} ,\\
		\Bscr^-_T(\hat v)&:=\{  g\in \Mat_m(\R) \colon \log \norm{g\, v}<-T \; \} .
	\end{align*}
	Since
	$$ \log \norm{g\, v}<-T \; \Leftrightarrow\; \norm{g\, v}^{-p} > e^{p\, T} $$
	by Markov inequality we have
	$$ \mu\left(\Bscr^-_T(\hat v)\right) \leq
	e^{-p\, T}\, \int \norm{g\, v}^{-p}\, d\mu(g)\leq  C\, e^{-p\, T} . $$
	Similarly, because
	$$ \log \norm{g}>T \; \Leftrightarrow\; \norm{g}^{p} > e^{p\, T} $$
	by Markov inequality we have
	$$ \mu\left(\Bscr^+_T(\hat v)\right) \leq
	e^{-p\, T}\, \int \norm{g}^{p}\, d\mu(g)\leq  C\, e^{-p\, T} . $$
	These two inequalities provide the stated upper bound.
\end{proof}

\begin{corollary}
	\label{coro bound BT(v)}
	Given $\mu\in \Mcal$, $n\in \N$  and $\hat v\in \Pp(\R^m)$,
	$$\mu^n \left( \Bscr_{n\,T}(\hat v) \right) \, \leq 2\,C^n\, e^{-p\, n\, T} . $$
\end{corollary}

\begin{proof}
By Proposition~\ref{Theta_n(mu n)},
$\mu^n\in \mathcal{M}^p_{C^n}$. The conclusion follows by Lemma~\ref{bound BT(v)}.
\end{proof}

\begin{lemma}
	\label{bound int BT(v) log norm gv}
	Given $\mu\in \Mcal$ and $\hat v\in \Pp(\R^m)$,
\begin{align*}
\int_{\Bscr_T(\hat v)} \vert \log \norm{g\, v}\vert \, d\mu(g)   &\leq 3\, \frac{C}{p}\, e^{-\frac{p\, T}{2}}  ,  \\
\int_{\Bscr_T(\hat v)}  \log^2 \norm{g\, v}  \, d\mu(g)   &\leq 10\,\frac{ C}{p^2}\, e^{-\frac{p\, T}{2}}  
\end{align*}
\end{lemma}

\begin{proof}
	Using trivial inequality $x^2\leq 2\, e^{\vert x\vert}$ we get
	that $x^2\leq \frac{2}{p^2}\, e^{p\, \vert x\vert}$
	\begin{align*}
		\int \log^2\norm{g\, v}\, d\mu(g)
		& \leq \frac{2}{p^2}\, \int e^{p\, \vert \log \norm{g\, v}\vert}\, d\mu(g)\\
		&\leq 
		\frac{2}{p^2}\, \int \max\left\{ \norm{g}^p, \frac{1}{\norm{g v}^p} \right\}\, d\mu(g)\\
		&\leq 
		\frac{2}{p^2}\, \int   \norm{g}^p  +  \frac{1}{\norm{g v}^p}  \, d\mu(g)
		\leq \frac{4 \, C}{p^2} .
	\end{align*}
Similarly, since $x^4/4!\leq e^{|x|}$,
we have that $x^4\leq \frac{24}{p^4}\, e^{p\, |x|}$, which gives
$$	\int \log^4\norm{g\, v}\, d\mu(g)
 \leq \frac{24}{p^4}\, \int \max\left\{ \norm{g}^p, \frac{1}{\norm{g v}^p} \right\}\, d\mu(g)  \leq \frac{48\, C}{p^4} . $$
Hence
	by Cauchy-Schwarz 
	\begin{align*}
		\int_{\Bscr_T(\hat v)} \vert \log \norm{g\, v}\vert \, d\mu(g)  &\leq 
		\sqrt{\mu\left(\Bscr_T(\hat v)\right) }\, \left(\int \log^2\norm{g\, v}\, d\mu(g) \right)^{1/2}\\
		&\leq \sqrt{2\, C\, }\, e^{-\frac{p\, T}{2}}\, \frac{2\, \sqrt{C}}{p} < \frac{3\, C}{p}\, e^{-\frac{p\, T}{2}}  
	\end{align*}
and in the same way
$$  \int_{\Bscr_T(\hat v)} \log^2 \norm{g\, v} \, d\mu(g)   \leq \sqrt{2\, C\, }\, e^{-\frac{p\, T}{2}}\, \sqrt{ \frac{48\, C}{p^4} } < 10\,\frac{C}{p^2}\, e^{-\frac{p\, T}{2}} . $$
\end{proof}

\begin{corollary}
	\label{coro bound int BT(v) log norm gv}
	Given $\mu\in \Mcal$, $n\in\N$ and $\hat v\in \Pp(\R^m)$,
\begin{align*}
	\int_{\Bscr_{n\, T}(\hat v)} \vert \log \norm{g\, v}\vert \, d\mu^n(g)   & \leq \frac{3\, C^n}{p}\, e^{-\frac{p\, n\, T}{2}}  , \\
	\int_{\Bscr_{n\, T}(\hat v)}  \log \norm{g\, v}^2  \, d\mu^n(g)   & \leq \frac{10\, C^n}{p^2}\, e^{-\frac{p\, n\, T}{2}}  .	
\end{align*}
\end{corollary}

\begin{proof}
	By Proposition~\ref{Theta_n(mu n)},
	$\mu^n\in \mathcal{M}^p_{C^n}$. The conclusion follows by Lemma~\ref{bound int BT(v) log norm gv}.
\end{proof}

The following bounds on exponential moments hold.

\begin{lemma}
	\label{bounds on Theta j}
	Given $g\in \Mat_m(\R)$ and
	linear subspaces $E^j\subset F^k\subseteq \R^m$ of dimensions $1\leq j <k\leq m$,
	$$\det(g_{\vert E})\geq \left[ \frac{ \det(g_{\vert F})  }{\norm{g}^{j}}\right]^{\frac{j}{k-j}}.$$
	In particular for any $\mu\in \Prob(\Mat_m(\R)$,\; 
	$$ \underline\Theta^j_p(\mu)\leq
	\sqrt{ \bar\Theta_{\frac{2 j^2 p}{k-j}}(\mu) } \,  \sqrt{ \underline\Theta^k_{\frac{2 p j}{k-j}}(\mu) }.$$
\end{lemma}

\begin{proof}
	If
	$s_1^F\geq s_2^F\geq \cdots \geq s_k^F\geq 0$
	are the singular values of $g_{\vert F}$ and $s_1^E\geq s_2^E\geq \cdots \geq s_j^E\geq 0$ the singular values of $g_{\vert E}$,
	for any $j\leq l\leq k$, \, 
	$s_l^F\leq s_j^E\leq (s_1^E\cdots s_j^E)^{1/j} =(\det(g_{\vert E}))^{1/j}$,
	which implies that
	$$ \det(g_{\vert F}) = s_1^F\cdots s_k^F\leq \norm{g}^j \,(\det(g_{\vert E}))^{\frac{k-j}{j}}   .$$
	This establishes the first inequality.
	
	Given $E\in \Gr_j(\R^m)$ take any $F\in \Gr_k(\R^m)$ that contains $E$. Then combining the previous with Cauchy-Schwarz inequality
	\begin{align*}
		\int \frac{1}{|\det(g_{\vert E})|^p }\, d\mu(g) &\leq 
		\int \frac{\norm{g}^{\frac{j^2 p}{k-j}} }{|\det(g_{\vert F})|^{\frac{j p}{k-j}} }\, d\mu(g)\\
		&\leq \left( \int \norm{g}^{\frac{2 j^2 p}{k-j}}\, d\mu(g)  \right)^{1/2}\, \left(\int \frac{1}{|\det(g_{\vert F})|^{\frac{2 j p}{k-j}} }\, d\mu(g) \right)^{1/2}\\
		&\leq \sqrt{ \bar\Theta_{ \frac{2 j^2 p}{k-j}}(\mu) } \,  \sqrt{ \underline\Theta^k_{\frac{2 p j}{k-j}}(\mu) } ,
	\end{align*}
	and taking the supremum in $E$, the second inequality holds.
\end{proof}

\section{Furstenberg Kifer Theory}
\label{FK}

In this section we prove an extended version of a classical result of H. Furstenberg and Y. Kifer~\cite[Theorem 3.5]{FK83}
from $\GL_m(\R)$ to $\Mat_m(\R)$-cocycles.

\begin{theorem}
	\label{non-invertible FK}
	Given $\mu\in \Mcal$ such that
	\begin{enumerate}
		\item $\mu$ is quasi-irreducible;
		\item $L_1(\mu)>L_2(\mu)$;
	\end{enumerate}
	then for $\mu^\N$-almost every $\omega\in\Mat_m(\R)^\N$ and $\hat v\in \Pp(\R^m)$,
	$$ 
	\lim_{n\to \infty} \frac{1}{n}\, \log \norm{ A^n(\omega)\, v}=L_1(\mu).$$
\end{theorem}

\subsection{Furstenberg Kifer's abstract theorem}

The proof relies on an abstract Birkhoff theorem for Markov processes, originally established by Furstenberg and Kifer~\cite{FK83}, which we present here without proof.

Let $M$ be a compact metric space and $k:M\to \Prob(M)$ a continuous map,
where we endow the space $\Prob(M)$ of Borel probability measures on $M$ 
with the weak * topology. The transition kernel $k$ determines a
Markov-Feller operator $\Qop_k:C(M)\to C(M)$ defined by
$$ (\Qop_k\varphi)(x):= \int \varphi(y)\, dk_x(y) .$$
A random process $\xi_n:\Omega\to M$, on some probability space $(\Omega,\Pp)$, is said to be a $k$-Markov process if for any Borel set $B\subseteq M$,
$$ \Pp[\, \xi_n \in B \, \vert\, \xi_{n-1}, \cdots, \xi_0\,]=  k_{\xi_{n-1}}(B) .$$
A probability measure $\eta\in \Prob(M)$ is called a $k$-stationary
if 
$$ \eta= \int k_x\, d\eta(x)=: \Qop_k^\ast \eta ,$$
where the right-hand-side expression stands for  the adjoint  $\Qop_k$.


\begin{theorem}[Furstenberg-Kifer]
	\label{abstract FK}
	Consider a $k$-Markov process $\{\xi_n\}_{n\geq 0}$ and a continuous function $\varphi\in C(M)$. Then
	\begin{enumerate}
		\item $\Pp$-almost surely
		$$  \limsup_{n\to \infty} \frac{1}{n}\, \sum_{j=0}^{n-1} \varphi(\xi_j(\omega))\leq \sup \left\{\, \smallint\varphi\, d\eta\colon \, \eta \text{ is } k\text{-stationary} \right\}.$$
		
		\item Given real numbers $\alpha\leq \beta$, if  for all $k$-stationary measures 
		$\eta\in\Prob(M)$, we have \, $\alpha\leq \smallint \varphi\, d\eta\leq \beta$   \, then $\Pp$-almost surely
		$$  \alpha\leq \liminf_{n\to \infty} \frac{1}{n}\, \sum_{j=0}^{n-1} \varphi(\xi_j(\omega))\leq  \limsup_{n\to \infty} \frac{1}{n}\, \sum_{j=0}^{n-1} \varphi(\xi_j(\omega))\leq \beta .$$
	\end{enumerate}
\end{theorem}

Notice that the first inequality in (2) is an easy corollary of (1) applied to the observable $-\varphi$.  See Theorems 1.1 and 1.4 in~\cite{FK83}.

\subsection{Proof of Theorem~\ref*{non-invertible FK} }
Fix the matrix space  $\Scal=\Mat_m(\R)$,  a probability $\mu\in \Prob(\Scal)$, set 
$M:=\Scal\times \Pp(\R^m)$ and then define
$k:M\to \Prob(M)$,
$k_{(g,v)}:=\mu\times \delta_{\hat g\, \hat v}$.
Consider also the observable
$\psi:M\to \R$, 
$\psi(g, \hat v):= \log \norm{g\, v}$.
The goal is to apply Theorem~\ref{abstract FK} to $M$, $k$, and $\psi$. Several technical difficulties arise, namely
\begin{itemize}
	\item $M$ is not compact because $\Scal$ is not;
	\item  $\hat g\, \hat v$ is undefined when
	$g$ is non-invertible and $v\in \ker(g)$. In particular the kernel
	$k_{(g,v)}=\mu\times \delta_{\hat g\, \hat v}$ is discontinuous at these points;
	\item  $\psi$  is undefined along the singularity set
	$$ \mathscr{N}:=\{(g, \hat v)\in \Scal\times \Pp(\R^m)\colon v\in \ker(g) \} .$$
	Moreover, the discontinuities of  $\psi$  are not removable because it has non trivial oscillations as $(g,\hat v)$
	approaches $\mathscr{N}$.	
\end{itemize}

Denote by $\Prob_\mu(\Pp(\R^m))$ the  space
of $\mu$-stationary measures. By definition,  $\eta\in \Prob(\Pp(\R^m))$ is $\mu$-stationary if and only if $\mu\times \eta$ is $k$-stationary.

\begin{proposition}
\label{space of stationary measures} 
 $\Prob_\mu(\Pp(\R^m))$ is  non-empty  compact and convex.\\
 Moreover, $\beta_\mu :\Prob_\mu(\Pp(\R^m)))\to \R$,
 $$\beta_\mu(\eta):= \iint  \psi(g, \hat v)\, d\eta(\hat v)\, d\mu(g) , $$
 is a bounded linear functional.
\end{proposition}

\begin{proof}
Because $\mu\in \Mcal$, the operator
$\Prob(\Pp(\R^m))\ni \eta \mapsto \Qop_\mu^\ast \eta\in \Prob(\Mat_m(\R))$
is well defined by
$$  \int\varphi\, d (\Qop_\mu^\ast \eta) := \int \!\!\underbrace{ \int \varphi(\hat g\, \hat v)\, d\mu(g) }_{\text{continuous in } \hat v} \, d\eta(\hat v) \qquad \forall\, \varphi\in C^0(\Projm). $$
Hence, 
 $\eta'=\Qop_\mu^\ast\eta$   satisfies the pre-condition in
Definition~\ref{def mu stationary} for being stationary, for every $\eta\in \Prob(\Pp(\R^m))$.
Moreover, any sub-limit of the sequence $\eta_n=\frac{1}{n}\, \sum_{j=0}^{n-1} \Qop_\mu^{\ast j}\eta_0$ is   $\mu$-stationary.
Therefore, the space $\Prob_\mu(\Pp(\R^m))$ is  non-empty -- it contains all sub-limits of $\eta_n$.

Compactness and convexity of the space of probability measures follow. Notice that $\Qop_\mu:C^0(\Projm)\to C^0(\Projm)$ is a bounded (hence continuous) linear operator.
The second statement is ensured by the moment conditions in the definition of $\Mcal$.
\end{proof}

\begin{lemma}
	\label{lemma  all beta equal}
	If $\mu\in \Mcal$ is quasi-irreducible then\,
	$\beta_\mu(\eta) =L_1(\mu)$ 
	for every   $\eta\in \Prob_\mu(\Pp(\R^m))$.
\end{lemma}

\begin{proof}
	Assume $\beta_\mu(\eta)<L_1(\mu)$ and replace $\eta$ by an extremal minimizer of the functional $\beta_\mu$ in $\Prob(\Pp(\R^m))$, which is an ergodic $\mu$-stationary measure. This implies that $\mu\times \eta$ is an ergodic $k$-stationary measure as well as an invariant ergodic measure of 
	$F:\Scal^\N\times \Pp(\R^m)\to \Scal^\N\times \Pp(\R^m)$,
	$F(\omega, \hat v):=(\sigma\omega, A(\omega)\, \hat v)$, where $A(\omega)=\omega_0$. Hence by Birkhoff's ergodic theorem, for $\eta$-almost every $\hat v\in \Pp(\R^m)$ and $\mu^\N$ almost every $\omega$,
	\begin{align} \nonumber
		\lim_{n\to \infty} \frac{1}{n}\, \log\norm{A^n(\omega)\, v} &= 
		\lim_{n\to \infty} \frac{1}{n}\, \sum_{j=0}^{n-1} \psi(F^j(\omega, \hat v)) \\
		\label{Furstenberg formula}
		&= \iint \psi\, d\mu\times \eta = \beta_\mu(\eta).
	\end{align}
	Let $E$ be the linear span of $\supp(\eta)$, the support of $\eta$.
	This space is invariant under every $g\in \supp(\mu)$ because  for a $\mu$-stationary measure  $\hat g\, \supp(\eta)\subseteq \supp(\eta)$ (see~\cite[Lemma 3.4]{FK83}).
	For the sake of completeness we reproduce the argument here.
	Setting $\hat E=\{\hat v\colon v\in E\setminus\{0\}\}$,  
	$$ 1=\eta(\hat E) = \int \underbrace{ \eta(g^{-1}\hat E) }_{\leq 1} \, d\mu(g),  $$
	which implies that $\eta(g^{-1}\hat E)=1$ for $\mu$-almost every $g\in \supp(\mu)$. For such $g$ we have $g^{-1} E\subset E$ and hence $g \, E\subset E$. By continuity this inclusion must hold for all $g\in \supp(\mu)$.
	
	Since~\eqref{Furstenberg formula} holds for $\eta$-almost every $\hat v$ we can find a basis $(v_1, \ldots, v_k)$ of $E$ such that 
	~\eqref{Furstenberg formula} holds with $\hat v=\hat v_i$ for all $1\leq i\leq  k$ and $\mu$-almost every $g$. It follows that~\eqref{Furstenberg formula} holds  for all $ v\in E\setminus\{0\}$ and $\mu^\N$-almost every $\omega$.	
	If $\beta_\mu(\eta)<L_1(\mu)$ this contradicts the quasi-irreducibility of $\mu$, 	
	see Definition~\ref{def quasi-irreducibility}. 
	Therefore  $\beta_\mu(\eta)=L_1(\mu)$ for all $\eta\in \Prob_\mu(\Pp(\R^m))$.
\end{proof}

To address the technical difficulties raised above we do three things:
first we truncate $\psi$,  then we 
restrict $M=\Mat_m(\R)\times \Pp(\R^m)$ to a compact subset and finally we  induce $k$ away from its singularities to make it continuous. Let us now introduce the main players in the proof, all of which depend on a large parameter $T$.

First we define the truncation
$\psi_T:\Mat_m(\R)\times \Pp(\R^m)\to \R$ by
$$ \psi_T(g, \hat v):=\begin{cases}
	\log \norm{g\, v} & \text{ if } -T\leq \log \norm{g\, v}\leq  \log \norm{g} \leq T \\
	\phantom{-}T & \text{ if } \;  \log \norm{g}> \phantom{-}T \\
	-T & \text{ if } \; \log \norm{g\, v}<-T	
\end{cases} . $$
Consider the compact space
$$  M_T:=\left\{ (g,\hat v)\in \Mat_m(\R)\times\Pp(\R^m)\, \colon \, -T\leq \log \norm{g\, v}\leq \log \norm{g} \leq T\, \right\}  .$$
We will also  write
$$ M_T(\hat v):=\{g\in \Mat_m(\R)\colon (g,\hat v)\in M_T \} .$$
Next we induce the kernel $k$ on $M_T$,
defining the stochastic kernel  $\tilde k:M_T\to\Prob(M_T)$ where for any Borel set $B\subset M_T$,
\begin{align}
	\label{def tilde k}
	&\tilde k_{(g_0, \hat v)}(B) := \\
	& \quad \sum_{n=1}^\infty  \int  \int_{M_T(g^{n-2}  \hat v)^\complement}  \cdots \int_{M_T(g_0  \hat v)^\complement} {\bf 1}_B(g_n, g^{n-1}  \hat v) \, d\mu(g_1)\, \cdots\, d\mu(g_n),\nonumber
\end{align} 
using the shorthand notation 
$g^j:= g_j\, \cdots \, g_1\, g_0$.

\begin{proposition}
	The kernel $\tilde k:M_T\to\Prob(M_T)$  is continuous.
\end{proposition}

\begin{proof}
	It is enough to prove that given a continuous function 
	$f:M_T\to\R$, the function $\Qop_{\tilde k} f:M_T\to \R$,
	$(\Qop_{\tilde k} f)(g_0, \hat v)=\int f\, d \tilde k_{(g_0, \hat v)}$, 
	is also continuous.
	
	Given $(g_0, \hat v)\in M_T$, let for each $n\geq 1$  $\Bscr_n(g_0, \hat v)\subset \Mat_m(\R)^\N$ be the subset consisting of all sequences $\{g_j\}_{j\geq 1}$ such that
	\begin{itemize}
		\item $(g_n, g_{n-1} \, \cdots \, g_1 \, g_0 \, \hat v)\in B$,
		\item $g_j\notin M_T(g_{j-1}\, \cdots \,g_1 \, g_0\,\hat v)$\,
		$\forall\, 1\leq j\leq n-1$. 
	\end{itemize}
	With respect to the Bernoulli measure $\mu^\N$, these sets are pairwise disjoint $\pmod 0$.
	In particular 
	$$ \sum_{n=1}^\infty \mu^\N(\Bscr_n(g_0, \hat v))\leq  1 .$$
	Now, by definition $\Qop_{\tilde k} f=\sum_{n=1}^\infty h_n$, where   $h_n(g_0, \hat v)$ is defined to be
	$$ \int\!\cdots\!\int 
	f(g_n, g_{n-1}  \cdots  g_1  g_0  \hat v)\,
	{\bf 1}_{\Bscr_n(g_0, \hat v)}
	(g_1, \ldots, g_n)\,   d\mu(g_1) \cdots   d\mu(g_n) 
	.$$
	Notice the following key facts about the function $h_n$:
	\begin{itemize}
		\item The above integrand 
		$ H_n(g_n, \ldots, g_1,\, g_0, \hat v)$ is measurable in all variables but continuous in $(g_0, \hat v)$ for $\mu^n$-almost every (fixed) $(g_1\, \ldots, g_n)\in \Mat_m(\R)^n$.
		\item   $H_n $ is bounded in absolute value by $\norm{f}_\infty$ and by the dominated convergence theorem
		$(g_0, \hat v)\mapsto h_n(g_0, \hat v)$ is continuous.
		
		\item   Because of the previous item,  $|h_n |\leq \norm{f}_\infty\, \mu^\N(\Bscr_n(g_0, \hat v))$.
	\end{itemize}
	Therefore by Weierstrass M-test $\sum_{n=1}^\infty h_n$ is a continuous function.
\end{proof}

Consider also:
\begin{itemize}
	\item the space $\Omega:=\Mat_m(\R)^\N\times \Pp(\R^m)$, 
	\item  the random process $\xi_n:\Omega\to \Mat_m(\R)\times \Pp(\R^m))$,  defined by
	$\xi_n(\omega, \hat v) :=  (\omega_n, A^n(\omega)\, \hat v)$,
	\item the cylinder
	$$ \Omega_T:=\left\{ (\omega, \hat v)\in \Omega \colon (\omega_0, \hat v)\in M_T \right\}, $$
	\item the \textit{induced map} $\tilde F:\Omega_T\to \Omega_T$    defined by
	$$\tilde F(\omega, \hat v):=(\sigma^{\tau(\omega, \hat v)}(\omega), A^{\tau(\omega, \hat v)}(\omega)\, \hat v ) ,$$
	where
	$\tau(\omega, \hat v)  := \min\{j>0 \colon F^j(\omega, \hat v)\in \Omega_T \}$,
	\item  the sequence sojourn times  $\tau_n:=\tau\circ \tilde F^{n-1}$, where $n\geq 1$,
	\item the \textit{induced process} $\tilde \xi_n:\Omega\to M_T$, $\tilde \xi_n(\omega):= \xi_{\tau_n(\omega)}(\omega)$,
	\item for each  $\eta\in \Prob(\Pp(\R^m))$, the normalized restricted measures
	\begin{align*}
		\Pp_{\eta,T} &:=\frac{1}{(\mu\times \eta)(M_T)}\, \mu\times \eta\vert_{M_T}\in \Prob(M_T)\\
		\tilde \Pp_{\eta,T} &:=\frac{1}{(\mu\times \eta)(M_T)}\, \mu^\N\times \eta\vert_{\Omega_T}\in \Prob(\Omega_T)\\
	\end{align*}
\end{itemize}

\begin{proposition}
	Given $\eta\in \Prob_\mu(\Pp(\R^m))$, the following hold:
	\begin{enumerate}
		\item $\mu\times \eta$ is a $k$-stationary measure, and also ergodic if $\eta$ is ergodic,  
		
		\item  $\tilde \Pp_{\eta, T}$ is  $\tilde F$-invariant, and also ergodic if $\eta$ is ergodic,  
		
		\item  $\Pp_{\eta, T}$ is  $\tilde k$-stationary, and also ergodic if $\eta$ is ergodic,  
		
		\item Every ergodic $\tilde k$-stationary measure is equal to $\Pp_{\eta, T}$ for some  ergodic $\eta\in \Prob_\mu(\Pp(\R^m))$.
	\end{enumerate}

\end{proposition}

\begin{proof}
	If $\eta\in \Prob(\Pp(\R^m))$ is $\mu$-stationary then for every $f\in C(\Pp(\R^m))$,
	$$  \int f(\hat v)\, d\eta(\hat v)= \iint f(g\, \hat v)\, d\mu(g)\, d\eta(\hat v) .$$
	Thus for any $h\otimes f\in C(\Mat_m(\R)\times \Pp(\R^m))$, $(h\otimes f)(g,\hatv):=h(g)\, f(\hat v)$,
	\begin{align*}
		\int (h\otimes f)(g, \hat v)\, d(\mu\times \eta)(g, \hat v) &= \int h(g')\, d\mu(g')\, \int f(\hat v)\, d\eta(\hat v)\\
		&= \iiint h(g')\,f(g\, \hat v)\, d\mu(g)\, d\eta(\hat v)\,d\mu(g') \\
		&= \iiint (h\otimes f)(g', g\, \hat v)\,d\mu(g') \, d\mu(g)\, d\eta(\hat v) \\
		&= \iint (h\otimes f)(g', g\, \hat v)\,d\mu(g') \, d(\mu\times \eta)(g, \hat v).
	\end{align*} 
	Since the linear span of the  functions $h\otimes f$ is dense in $C(\Mat_m(\R)\times \Pp(\R^m))$ it follows that given $f:\Mat_m(\R)\times \Pp(\R^m)\to\R$ continuous,
	$$  \int f(g, \hat v)\, d(\mu\times \eta)(g, \hat v) =
	\iint f(g', g\, \hat v)\,d\mu(g') \, d(\mu\times \eta)(g, \hat v) , $$
	which implies that $\mu\times \eta$ is $k$-stationary.
	The converse implication, that all $k$-stationary measures have this form is also easy to verify. Hence the map
	$\eta\mapsto \mu\times \eta$ is an isomorphism that maps  $\Prob_\mu(\Pp(\R^m))$ onto the compact convex space of $k$-stationary measures. In particular this linear isomorphism  takes   extremal measures of one to extremal measures of the other compact convex space.
	Since these are the ergodic stationary measures, item (1) holds.
	
	By~\cite[Proposition 5.5]{Viana2014}, the measure $\mu\times \eta$ is invariant under the projective map
	$F:\Mat_m(\R)^\N\times \Pp(\R^m)\to \Mat_m(\R)^\N\times \Pp(\R^m)$.
	Moreover, from~\cite[Proposition 5.13]{Viana2014} we get  that $\mu\times \eta$ is ergodic w.r.t. $F$ when $\eta$ is an ergodic $\mu$-stationary measure.

	Since $F$ preserves $\mu^\N\times \eta$, the induced map $\tilde F:\Omega_T\to \Omega_T$
	preserves the normalized restricted measure 
	$\tilde \Pp_{\eta,T}$, see~\cite[Proposition 1.4.1]{OV2016}. Moreover, if $\mu\times \eta$ is ergodic w.r.t. $F$ then the induced map $\tilde F$ is also ergodic w.r.t. $\tilde \Pp_{\eta,T}$, see~\cite[Exercise 4.3.3]{Viana2014}. This proves item (2).

	To prove (3) and (4) let $\Mscr$  be the Banach space
	of finite signed measures on $M=\Mat_m(\R)\times\Pp(\R^m)$ with the total variation norm. The adjoint of the Markov operator $\Qop_k$
	acts as a bounded linear operator $\Qop_k^\ast:\Mscr\to \Mscr$, defined by $\Qop_k^\ast \tilde\eta := \int \mu\times \delta_{g\,\hat v}\, d\tilde\eta(g,\hat v) $.
	The  subspaces
	$\Mscr_T:=\{\tilde \eta\in \Mscr\colon \tilde\eta(M_T^\complement)=0 \}$ and
	$\Mscr_T^\complement:=\{\tilde \eta\in \Mscr\colon \tilde\eta(M_T)=0 \}$ are closed complementary linear subspaces, i.e., $\Mscr=\Mscr_T\oplus \Mscr_T^\complement$.
	This direct sum decomposition determines the
	projections $\pi_T:\Mscr\to \Mscr_T$ and
	$\pi_C:\Mscr\to \Mscr_T^\complement$, which act on measures by restriction. We use them to introduce the following matrix block structure for $\Qop_k^\ast$
	$$ \Qop_k^\ast = \left[ \begin{array}{c|c}
		\Qop^\ast_{TT} & \Qop^\ast_{TC}\\
		\hline
		\Qop^\ast_{CT} & \Qop^\ast_{CC}
	\end{array}\right] , $$
	where the block operators are defined by
	\begin{align*}
		\Qop^\ast_{TT}&:= \pi_T\circ \Qop^\ast_k\circ \pi_T,\\
		\Qop^\ast_{TC}&:= \pi_T\circ \Qop^\ast_k\circ \pi_C ,\\
		\Qop^\ast_{CT}&:= \pi_C\circ \Qop^\ast_k\circ \pi_T,\\
		\Qop^\ast_{CC}&:= \pi_C\circ \Qop^\ast_k\circ \pi_C  .
	\end{align*}
	The adjoint Markov operator $\Qop_k^\ast:\Mscr\to \Mscr$ has eigenvalue $1$ because there exists at least one stationary measure. On the other hand the block operator
	$\Qop^\ast_{CC}:\Mscr_T^\complement\to \Mscr_T^\complement$ has small spectral radius because  $\mu\in \Mcal$, which implies that the transition probability measures $\mu\times \delta_{g\, \hat v}$ are predominantly  concentrated in $M_T$.
	In particular $1$ lies outside the spectrum of $\Qop^\ast_{CC}$. The isospectral reduction of $\Qop_k^\ast$ at the eigenvalue $\lambda=1$, see~\cite[Equation (3.1)]{BDT22}, is defined as
	\begin{align*}
		\Rop &=\Qop^\ast_{TT} + \Qop^\ast_{TC}\, (I-\Qop^\ast_{CC})^{-1}\, \Qop^\ast_{CT} \\
		&=\Qop^\ast_{TT} + \sum_{n=0}^\infty \Qop^\ast_{TC}\, (\Qop^\ast_{CC})^{n}\, \Qop^\ast_{CT} .
	\end{align*}
	A straightforward comparison reveals that  $\tilde k_{(g, \hat v)}=\Rop(\delta_{(g, \hat v)})$. 
	Specifically, the first term in~\eqref{def tilde k} corresponds directly to the first term in the series expansion of  $\Rop(\delta_{(g, \hat v)})(B)$:
	$$  \int {\bf 1}_B(g_1, g\hat v)\, d\mu(g_1) =
	\int_B  1\, d\mu\times \delta_{g\, \hat v} = 
	\Qop^\ast_{TT}(\delta_{(g, \hat v)})(B) .$$
	Similarly, the second terms align as follows:
	$$   \int \int_{M_T(g\, \hat v)^\complement} {\bf 1}_B(g_2, g_1\,g\hat v)\, d\mu(g_1)\, d\mu(g_2) =\Qop^\ast_{TC}\Qop^\ast_{CT}(\delta_{(g, \hat v)})(B) .$$
	By continuing this process, the terms match term by term. 
	Consequently, when the measure-valued integrals below are interpreted in a weak sense, we obtain
	\begin{align*}
		\Qop^\ast_{\tilde k} \tilde \eta  &= \int \tilde k_{(g, \hat v)}\, d\tilde\eta(g, \hat v) =  \int \Rop(\delta_{(g, \hat v)}) \, d\tilde\eta(g, \hat v)\\ &=\Rop\left(\, \int \delta_{(g, \hat v)}\, d\tilde\eta(g, \hat v)\, \right)  = \Rop (\tilde \eta) .
	\end{align*} 
	Hence the isospectral reduced operator $\Rop$ is exactly the adjoint of the Markov operator $\Qop_{\tilde k}$.
	By item (3) of~\cite[Theorem 3]{BDT22}, $\Qop_{\tilde k}(\mu\times \eta\vert_{M_T})=\mu\times \eta\vert_{M_T}$,
	which implies that $\Pp_{\eta, T}=\frac{1}{(\mu\times \eta)(M_T)}\, \mu\times \eta\vert_{M_T}$ is a $\tilde k$-stationary measure. By item (4) of the same theorem,
	given a $\tilde k$-stationary measure 
	$\tilde \eta\in \Prob(M_T)$, there is a $k$-stationary measure on $M$,
	which as we have seen above must be a multiple of $\mu\times \eta$ for some $\eta\in \Prob_\mu(\Pp(\R^m))$, such that $\tilde \eta$ is (up to a scaling factor) a restriction
	of $\mu\times \eta$. This shows that $\tilde \eta=\Pp_{\eta,T}$.
	
	Denoting by $\Prob_k(M)$, resp. $\Prob_{\tilde k}(M_T)$, sets of $k$-stationary, resp. $\tilde k$-stationary measures, these two compact convex spaces are isomorphic to $\Prob_\mu(\Pp(\R^m))$ via the isomorphisms:
	\begin{itemize}
		\item $\Prob_\mu(\Pp(\R^m))\ni \eta \mapsto \mu\times \eta\in \Prob_k(M)$,
		\item $\Prob_k(M)\ni \tilde \eta \mapsto  \tilde \eta(M_T)^{-1}\,\tilde \eta\vert_{M_T}\in \Prob_{\tilde k}(M_T)$.
	\end{itemize}
	Since isomorphisms between compact convex spaces preserve extremal points, a $\mu$-stationary measure $\eta$ is ergodic if and only if $\Pp_{\eta, T}$ is an ergodic $\tilde k$-stationary measure on $M_T$. This completes the proof of items (3) and (4).
\end{proof}

\begin{remark}
	$\{\tilde \xi_n\}_{n\geq 0}$ is a stationary  $\tilde k$-Markov process on $\left( \Omega_T, \tilde \Pp_{\eta, T}\right)$.	While we omit the proof since it is not needed later, this observation provides additional insight into the induced Markov kernel  $\tilde k$.
\end{remark}

The last ingredient for the proof is the observable
$\tilde \psi:\Omega\to \R$,
$$ \tilde \psi(\omega, \hat v):= \sum_{j=0}^{\tau(\omega,\hat v)-1} \psi(F^j(\omega, \hat v)) . $$

\begin{proposition}
	\label{prop bound L1 mu}
	Given $\mu\in \Mcal$, \;  $|L_1(\mu)|\leq \frac{C-1}{p}$.	
\end{proposition}

\begin{proof}
	We are going to use the inequality $p\, |x|\leq e^{p\, |x|}-1$.
	Take any measure  $\eta\in \Prob_\mu(\Pp(\R^m))$. Then
	\begin{align*}
		p\,|L_1(\mu)| &\leq \iint p\,|\log \norm{g\, v}|\, d\mu(g)\, d\eta(\hat v) \\
		&\leq -1+ \iint e^{ p\,|\log \norm{g\, v}|}\, d\mu(g)\, d\eta(\hat v) \\
		&\leq -1+ \iint \max\left\{\norm{g}^p, \frac{1}{\norm{g\, v}^p}  \right\}\, d\mu(g)\, d\eta(\hat v) 
		\leq -1+ C .
	\end{align*}
\end{proof}	

\begin{proposition}
	\label{inducing relation}
	For all $(\omega, \hat v)\in \Omega$ and $n\geq 1$, writing $\tau_n=\tau_n(\omega, \hat v)$,
	$$ \frac{1}{\tau_n}\, \log \norm{A^{\tau_n}(\omega)\, v} = \frac{n}{\tau_n}\cdot \frac{1}{n}\,\sum_{j=0}^{n-1} \tilde \psi(\tilde F^j(\omega, \hat v)) .$$
	Moreover, for any $\eta\in \Prob_\mu(\Pp(\R^m))$, 
	$$ \left\vert \int   \tilde\psi\, d\tilde \Pp_{\eta, T}-L_1(\mu) \right\vert \leq \frac{2\, C^2}{p}\, e^{-p\, T} . $$
\end{proposition}

\begin{proof}
	First notice that for all $(\omega, \hat v)\in \Omega$,
	$$
	\log \norm{A^{\tau_n(\omega, \hat v)}(\omega)\, v} = \sum_{j=0}^{\tau_n-1} \psi(F^j(\omega, \hat v))  = \sum_{j=0}^{n-1} \tilde \psi(\tilde F^j(\omega, \hat v)) .
	$$
	Dividing by $\tau_n$ we obtain the first statement.
	
	Given $\eta\in \Prob_\mu(\Pp(\R^m))$, applying Birkhoff's ergodic theorem to the projective map $F$ and the invariant measure $\tilde \Pp_\eta:=\mu^\N\times\eta$, we get
	$$ \lim_{n\to \infty} \frac{\tau_n(\omega, \hat v)}{n}=\tilde \Pp_\eta(\Omega_T)= (\mu\times \eta)(M_T) .$$
	Also, if $\eta$ is ergodic then the integral of the left-hand-side
	$$ \int_{\Omega_T}  \frac{1}{\tau_n}\, \log \norm{A^{\tau_n}(\omega)\, v}\, d\tilde \Pp_\eta(\omega, \hat v) = \int_{\Omega_T}  \frac{1}{\tau_n}\, \sum_{j=0}^{\tau_n-1}  \psi(F^j\omega, \hat v))\, d\tilde \Pp_\eta(\omega, \hat v) $$
	by Lemma~\ref{lemma  all beta equal} converges to 
	$$\int_{\Omega_T}\beta_\mu(\eta)\, d\tilde\Pp_\eta=\tilde\Pp_\eta(\Omega_T)\, \beta_\mu(\eta)=(\mu\times \eta)(M_T)\, L_1(\mu).$$
	On the other hand  
	$$ \int_{\Omega_T} \frac{n}{\tau_n}\cdot \frac{1}{n}\,\sum_{j=0}^{n-1} \tilde \psi(\tilde F^j(\omega, \hat v)) \, d\tilde\Pp_\eta(\omega, \hat v) \longrightarrow \frac{1}{(\mu\times \eta)(M_T)}\, \int_{\Omega_T}  \tilde \psi \,d\tilde\Pp_\eta(\omega, \hat v) $$
	and the limit in the right-hand-side is equal to 
	$\int\tilde \psi\, d\tilde \Pp_{\eta,T}$.
	This proves that $\int\tilde \psi\, d\tilde \Pp_{\eta,T}=(\mu\times \eta)(M_T)\, L_1(\mu)$ and
	by Proposition~\ref{prop bound L1 mu} and Lemma~~\ref{bound BT(v)}
	\begin{align*}
		\left\vert \int\tilde \psi\, d\tilde \Pp_{\eta,T} - L_1(\mu) \right\vert &\leq |L_1(\mu)|\, |(\mu\times \eta)(M_T)-1|\\
		&\leq |L_1(\mu)|\, \int \mu(M_T(\hat v)^\complement ) \,d\eta(\hat v)\leq 2\, \frac{C^2}{p}\, e^{-p\, T} .
	\end{align*}
\end{proof}

\begin{proposition}
	\label{prop tilde psi - psi T proximity}
	For $(\omega,\hat v) \in \Omega_T\cap F^{-1}(\Omega_T)$,\, 
	$\tilde \psi(\omega, \hat v) =\psi_T(\omega, \hat v) $,
	and on the complementary set 
	$\Sigma_T:=\Omega_T\setminus  F^{-1}(\Omega_T)$,
	there exists a positive number  $T_0=T_0(C,p)$ such that for all $T\geq T_0 $,
	$$ \int \vert \tilde \psi-\psi_T\vert \, d\tilde\Pp_{\eta, T}\leq
	\int_{\Sigma_T}  ( \vert \tilde \psi\vert + T )\, d\tilde\Pp_{\eta, T}  < e^{-\frac{p\, T}{3}}.$$
\end{proposition}

\begin{proof}
	Given  $n\geq 1$ let  $\Bscr_n \subset \Omega_T$ be the  set of    $(\omega, \hat v)\in\Omega_T$ such that
	\begin{itemize}
		\item $(\omega_n, A^{n}(\omega) \, \hat v)\in M_T$,
		\item $(\omega_j, A^j(\omega)\, \hat v) \notin M_T $\;
		$\forall\, 1\leq j\leq n-1$. 
	\end{itemize}
	By construction $\Bscr_n=\left\{(\omega, \hat v)\in\Omega_T\colon \, \tau(\omega, \hat v)=n \,  \right\}$ and these sets form a partition  of $\Sigma_T$ $\pmod 0$.
	For each $n\geq 1$, $\int_{\Bscr_n} |\tilde \psi|\, d\tilde \Pp_{\eta} $ is a sum of $n$ terms of the form
	$$ \int \!\cdots \! \int_{\Bscr_n}  \left\vert \log \norm{ \omega_i\, \frac{A^i(\omega)\, v}{\norm{A^i(\omega)\, v}}}\right\vert \, d\mu(\omega_i)\, \prod_{j\neq i} d\mu(\omega_j) $$
	Using lemmas~\ref{bound BT(v)} and~\ref{bound int BT(v) log norm gv} we get the bound
	$$  \int_{\Bscr_n} |\tilde \psi|\, d\tilde \Pp_{\eta} \leq n\, \frac{3\, C}{p}\, e^{-\frac{p\, T}{2}} \, \left(2\, C\, e^{-p\, T} \right)^{n-1} .$$
	Then, since $\sum_{n=1}^\infty n\, x^{n-1}= (1-x)^{-2}$,
	adding up we get
	\begin{align*}   \int_{\Sigma_T} |\tilde \psi|\, d\tilde \Pp_{\eta}  &\leq \left(\frac{3\, C}{p}\, e^{-\frac{p\, T}{2}} \right)\, \left(1-2\, C\, e^{-p\, T} \right)^{-2} \\
		&< \frac{3\cdot 5^2\cdot C}{4^2\, p}\, e^{-\frac{p\, T}{2}} \ll  e^{-\frac{p\, T}{3}} ,
	\end{align*} 
	when $T$ is sufficiently large. Since by Lemma~\ref{bound BT(v)} ,
	$$  \int_{\Sigma_T} T\, d\tilde \Pp_\eta \leq 2\, C\, T\, e^{-p\, T} \ll e^{-\frac{p\, T}{3}} , $$
	these upper bounds imply the stated inequality.
\end{proof}

\begin{remark}
	A   threshold for Proposition~\ref{prop tilde psi - psi T proximity}  is given by
	$$T_0(C,p):=\max\left\{ \frac{\log(10)+C}{p} , \, \frac{6}{p}\, \left[  \log\left( \frac{5}{p} \right) +\log C \right] \right\}.$$
\end{remark}
%
%
%
%
%
%

We are now ready to conclude the proof. To simplify the notation, we uniformly bound all the upper estimates from Propositions~\ref{inducing relation},\ref{prop tilde psi - psi T proximity},~\ref{bound BT(v)} and ~\ref{bound int BT(v) log norm gv} by
$$ \varepsilon(T):= e^{-\frac{p\, T}{3}} . $$

From Propositions~\ref{inducing relation} and~\ref{prop tilde psi - psi T proximity}, for every
$\eta\in \Prob_\mu(\Pp(\R^m))$, we have
$$ L_1(\mu)-2\, \varepsilon(T)\leq \int \psi_T\, d\tilde \Pp_{\eta,T}\leq L_1(\mu)+2\, \varepsilon(T).$$
Thus, by item (2) of Theorem~\ref{abstract FK}, for any $\hat v\in \Pp(\R^m)$, $\mu^\N$-almost surely,
\begin{align*}
	L_1(\mu)-2\,\varepsilon(T) &\leq \liminf_{n\to\infty} \frac{1}{n}\, \sum_{j=0}^{n-1}  \psi_T(\tilde F^j(\omega, \hat v)) \\
	& \leq \limsup_{n\to\infty} \frac{1}{n}\, \sum_{j=0}^{n-1}  \psi_T(\tilde F^j(\omega, \hat v))  \leq L_1(\mu)+ 2\, \varepsilon(T) .  
\end{align*}
By Birkhoff's ergodic theorem, the difference
$$  \frac{1}{n}\, \sum_{j=0}^{n-1}\tilde\psi(\tilde F^j(\omega, \hat v))- \frac{1}{n}\, \sum_{j=0}^{n-1} \psi_T(\tilde F^j(\omega, \hat v)) $$
converges pointwise to a constant function that is bounded in absolute value by
$\varepsilon(T)$. 
Hence, we obtain
\begin{align*}
	L_1(\mu)-3\, \varepsilon(T) &\leq \liminf_{n\to\infty} \frac{1}{n}\, \sum_{j=0}^{n-1} \tilde\psi(\tilde F^j(\omega, \hat v)) \\
	& \leq \limsup_{n\to\infty} \frac{1}{n}\, \sum_{j=0}^{n-1} \tilde\psi(\tilde F^j(\omega, \hat v))  \leq L_1(\mu)+3 \, \varepsilon(T) .  
\end{align*}
Applying Proposition~\ref{inducing relation}, we further deduce that 
\begin{align*}
	L_1(\mu)-4\, \varepsilon(T) &\leq \liminf_{n\to\infty} \frac{1}{n}\, \log \norm{A^n(\omega)\, v}  \\
	& \leq \limsup_{n\to\infty} \frac{1}{n}\, \log \norm{A^n(\omega)\, v}  \leq L_1(\mu)+4\, \varepsilon(T) .  
\end{align*}
Finally, letting $T\to \infty$,  we conclude that
$$\lim_{n\to\infty} \frac{1}{n}\, \log \norm{A^n(\omega)\, v} = L_1(\mu),$$
which completes the proof of Theorem~\ref{non-invertible FK}.

\section{Spectral Theory of Markov operator}\label{Markov}
In this section we establish the uniform ergodicity (strong mixing) of the Markov operator acting on a class of $\alpha$-H\"older  observables on the projective space, where $0<\alpha < p < 1$ is   sufficiently small. Throughout the section we assume that $\mu\in \Mcal$ satisfies the hypothesis of Theorem~\ref{thm non invertible Holder LE }, i.e., (1) $\mu$ is quasi-irreducible and (2) $L_1(\mu)>L_2(\mu)$.

Consider the Markov operator $Q_\mu: L^\infty(\Projm)\rightarrow L^\infty(\Projm)$ where $$(Q_\mu\varphi)(\hat{v}):=\int\varphi(\hat{g}\hat{v}) \ d\mu(g).$$

Let us introduce a metric on $\Projm$, namely, $$d(\hat{v},\hat{w}):=\frac{\norm{v\wedge w}}{\norm{v}\norm{w}}=\sin\angle(v,w),$$ where $v\in\hat{v}$ and $w\in\hat{w}$. Given $0<\alpha\leq 1$, we introduce the Banach algebra $(\Holal(\Projm),\norm{\cdot}_\alpha)$ where $$\Holal(\Projm):=\{\varphi\in L^\infty(\Projm):v_\alpha(\varphi)<\infty\},$$
and $v_\alpha(\cdot)$ stands for the $\alpha$-H\"older seminorm on $L^\infty(\Projm)$
 $$v_\alpha(\varphi):=\sup_{\hat{v}\neq \hat{w}} \frac{\abs{\varphi(\hat{v})-\varphi(\hat{w})}}{{d(\hat{v},\hat{w})}^\alpha}\; .$$ 
 The sub-multiplicative algebra norm  $\norm{\cdot}_\alpha$ is given by $$\norm{\varphi}_\alpha:=\norm{\varphi}_\infty + v_\alpha(\varphi)$$ 
 where the first term stands for the canonical sup norm on $L^\infty(\Projm)$, $$\norm{\varphi}_\infty:=\sup_{\hat{v}\in\Projm}\abs{\varphi(\hat{v})} .$$  

Notice that $v_\alpha$ is not a norm because $v_\alpha(\varphi)=0$ only implies that $\varphi$ is constant but not necessarily equal to $0$.

We will restrict the action of the Markov operator to H\"older observables in $\Holal(\Projm)$, i.e., $ Q_\mu:\Holal(\Projm)\to \Holal(\Projm)$.
A spectral feature of this operator is the quantity $$\kappa_\alpha(\mu):=\sup_{\hat{v}\neq \hat{w}}\int_{\Mat_m(\R)}\left[{\frac{d(\hat{g}\hat{v},\hat{g}\hat{w})}{d(\hat{v},\hat{w})}}\right]^\alpha\ d\mu(g),$$ 
that measures the $\mu$-average H\"older constant of the maps $\varphi_g\colon \hat{v}\mapsto \hat{g}\hat{v}$.

\begin{proposition}\label{7} Given 
$\varphi\in\Holal(\Projm)$, $v_\alpha(Q_\mu(\varphi))\leq \kappa_\alpha(\mu)v_\alpha(\varphi)$
\end{proposition}

\begin{proof}
Let $\varphi\in\Holal(\Projm)$. $\forall\hat{v},\hat{w}\in\Projm$ we have 
\begin{align*}
\frac{\abs{Q_\mu(\varphi)(\hat{v})-Q_\mu(\varphi)(\hat{w})}}{d(\hat{v},\hat{w})^\alpha}&= \frac{\abs{\int \varphi(\hat{g}\hat{v})-\varphi(\hat{g}\hat{w}) \ d\mu(g)}}{d(\hat{v},\hat{w})^\alpha}\\
&\leq \frac{\int \abs{ \varphi(\hat{g}\hat{v})-\varphi(\hat{g}\hat{w})} \ d\mu(g)}{d(\hat{v},\hat{w})^\alpha}\\
&\leq \int \frac{\abs{\varphi(\hat{g}\hat{v})-\varphi(\hat{g}\hat{w})}}{d(\hat{g}\hat{v},\hat{g}\hat{w})^\alpha}\cdot\frac{d(\hat{g}\hat{v},\hat{g}\hat{w})^\alpha}{d(\hat{v},\hat{w})^\alpha} \ d\mu(g)\\
&\leq v_\alpha(\varphi)\int \frac{d(\hat{g}\hat{v},\hat{g}\hat{w})^\alpha}{d(\hat{v},\hat{w})^\alpha} \ d\mu(g)\\
\end{align*}
The result follows from taking the supremum over $\hat{v}\neq\hat{w}$ on both sides. 
\end{proof}

\begin{lemma}
	\label{prop LE(vn)}
    Let $(v_n)_n$ be a sequence in $\Projm$ which converges to a non-zero vector $v_0$. Then $$\lim_{n\rightarrow+\infty}\frac{1}{n}\log\norm{g_n\cdots g_1 v_n}=L_1(\mu)$$ for $\mu^\mathbb{N}$ almost every $\{g_n\}_{n\in\mathbb{N}}$.
\end{lemma}

\begin{proof}
Given $g_1, \ldots, g_n\in \Mat_m(\R)$, let $u_n$ be the most expanded unit vector by the matrix $g_n\cdots g_1$. Then, see~\cite[Proposition 2.2]{DK-book},
    $$\frac{\norm{g_n\cdots g_1 v_n}}{\norm{g_n\cdots g_1}}\geq\left|\langle v_n , u_n\rangle\right|. $$  
   Taking the liminf in both sides, $$\liminf_{n\rightarrow +\infty}\frac{\norm{g_n\cdots g_1 v_n}}{\norm{g_n\cdots g_1}}\geq\left|\langle v_0, u_\infty \rangle\right|$$ for $\mu^\mathbb{N}$ almost every $\{g_n\}_{n\in\mathbb{N}}$, where $u_\infty=\lim_{n\rightarrow+\infty} u_n$.
   By~\cite[Proposition 4.4]{DK-book} this limit exists $\mu^\N$-almost surely.
    Notice that if $\langle v_0, u_\infty \rangle=0$ then $v_0\in u^\perp_\infty$,
    which by~\cite[Proposition 4.15]{DK-book} is the direct sum of the factors in Oseledets decomposition associated with Lyapunov exponents $<L_1(\mu)$. Hence, since $\mu$ is quasi-irreducible, by Theorem~\ref{non-invertible FK} the relation
    $v_0\in u^\perp_\infty$ can only hold  for a zero measure set of sequences $\{g_n\}_{n\in\mathbb{N}}$ with respect to the probability measure $\mu^\mathbb{N}$.
    Therefore, $$\liminf_{n\rightarrow +\infty}\frac{\norm{g_n\cdots g_1 v_n}}{\norm{g_n\cdots g_1}}>0 \qquad  \text{ for } \, \mu^\mathbb{N}\text{-a.e.} \, \{g_n\}_{n\in\mathbb{N}} . $$  Furthermore $$\lim_{n\rightarrow+\infty}\frac{1}{n}\log \frac{\norm{g_n\cdots g_1 v_n}}{\norm{g_n\cdots g_1}}=0 \qquad \text{ for } \, \mu^\mathbb{N}\text{-a.e. }\, \{g_n\}_{n\in\mathbb{N}} .$$ 
    Thus, $\mu$-almost surely  $$\lim_{n\rightarrow+\infty}\frac{1}{n}\log\norm{g_n\cdots g_1 v_n}=\lim_{n\rightarrow+\infty}\frac{1}{n}\log\norm{g_n\cdots g_1}=L_1(\mu) . $$  
\end{proof}

\begin{lemma}\label{unif}
	Given $\hat v\in \Projm$,
     $$\lim_{n\to \infty}  \frac{1}{n}\int \log\norm{gv} \ d\mu^n(g) =  L_1(\mu)$$ with uniform convergence  in $\hat v\in \Projm$.
\end{lemma}

\begin{proof}
Since $\Projm$ is compact, we only need to show that if a sequence in $\Projm$, say $(v_n)_n$, converges, then $$\lim_{n\rightarrow +\infty} \frac{1}{n}\int\log\norm{gv_n}d\mu^n(g)=L_1(\mu).$$ 
In fact, if the convergence were not uniform there would exist a sequence $\hat v_n\in \Projm$,
say converging to some $\hat v_0$, such that
the above limit would be different from $L_1(\mu)$. On the other hand, by Lemma~\ref{prop LE(vn)},  for any convergent sequence $(\hat v_n)_n$,
 $$\lim_{n\rightarrow +\infty}\frac{1}{n}\log\norm{g_n\cdots g_1 v_n}=L_1(\mu)\qquad \mu^\N\text{-a.s.} $$ 
The conclusion follows by Vitali's convergence theorem, integrating both sides and taking the limit as $n\to \infty$.
For that purpose it is enough proving that the sequence $\bigg\{\frac{1}{n}\log\norm{g_n\cdots g_1v_n}\bigg\}_{n\geq 1}$ is uniformly integrable w.r.t $\mu^\N$.

Given a large  $T\gg 0$,
recall from Section~\ref{Setting} that given $\hat v\in \Projm$,
	\begin{align*}
	\Bscr_T(\hat v)&:=\{ g \in \Mat_m(\R) \colon \log \norm{g} > T \; \text{ or }\;  \log \norm{g\, v}<-T \}  .
\end{align*}
Then $\left\{\, \left|\frac{1}{n}\log\norm{gv_n}\right|\geq T\, \right\}\subseteq  	\Bscr_{n\, T}(\hat v_n)$
and by corollaries~\ref{coro bound BT(v)} and~\ref{coro bound int BT(v) log norm gv},
\begin{align*}
\int_{\left|\frac{1}{n}\log\norm{gv_n}\right|\geq T}  & \left|\frac{1}{n}\log\norm{gv_n}\right| d\mu^n(g) 
\leq \frac{1}{n}\, \int_{	\Bscr_{n\, T}(\hat v_n)} \left| \log\norm{gv_n}\right|  d\mu^n(g)\\
&\leq \frac{1}{n}\, \left(\mu^n(\Bscr_{n\, T}(\hat v_n))\right)^{1\over 2}\, \left( \int    \log^2\norm{gv_n} \, d\mu^n(g) \right)^{1\over 2} \\
&\leq \frac{1}{n}\, \sqrt{2\, C^n\, e^{-n\, p\, T}}\, \sqrt{\frac{10\, C^n}{p^2}\, e^{-\frac{n\, p\, T}{2}} } <\frac{5\, C^n}{n\, p}\, e^{-\frac{3}{4}\, n\, p\, T } .
\end{align*}
Finally, as $T\to +\infty$, we have $C\, e^{-\frac{p\, T}{2}}<1$ and
$$\sup_n\int_{\left|\frac{1}{n}\log\norm{gv_n}\right|\geq T} \left|\frac{1}{n}\log\norm{gv_n}\right| d\mu^n(g)  <\sup_n \frac{5\, C^n}{n\, p}\, e^{-\frac{n\, p\, T}{2}   } < \frac{5}{p}\, C\, e^{-\frac{p\, T}{2}}\to 0$$
which  proves the uniform integrability and concludes the proof.
\end{proof}

\begin{proposition}
	\label{spectral property}
Given $\mu\in\mathcal{M}^p_C$ and  $m,n\in\N$, $$\kappa_\alpha(\mu^{m+n})\leq \kappa_\alpha(\mu^m)\, \kappa_\alpha(\mu^n) .$$ 
\end{proposition}

\begin{proof}
We have
\begin{align*}
\kappa_\alpha(\mu^m*\mu^n)&=\sup_{\hat{v}\neq\hat{w}}\int \left[\frac{d(\hat{g}\hat{v},\hat{g}\hat{w})}{d(\hat{v},\hat{w})}\right]^\alpha \ d(\mu^m*\mu^n)(g)\\
&=\sup_{\hat{v}\neq\hat{w}}\int \int \left[\frac{d(\hat{g}\hat{h}\hat{v},\hat{g}\hat{h}\hat{w})}{d(\hat{v},\hat{w})}\right]^\alpha \ d\mu^m(g) \ d\mu^n(h)\\
&=\sup_{\hat{v}\neq\hat{w}}\int \int \left[\frac{d(\hat{g}\hat{h}\hat{v},\hat{g}\hat{h}\hat{w})\, d(\hat{h}\hat{v},\hat{h}\hat{w})}{d(\hat{h}\hat{v},\hat{h}\hat{w})\, d(\hat{v},\hat{w})}\right]^\alpha \ d\mu^m(g) \ d\mu^n(h)\\
&\leq \sup_{\hat{v}\neq\hat{w}}\int \left[\frac{d(\hat{h}\hat{v},\hat{h}\hat{w})}{d(\hat{v},\hat{w})}\right]^\alpha \int \left[\frac{d(\hat{g}\hat{h}\hat{v},\hat{g}\hat{h}\hat{w})}{d(\hat{h}\hat{v},\hat{h}\hat{w})}\right]^\alpha  \ d\mu^m(g) \ d\mu^n(h)\\
&\leq \kappa_\alpha(\mu^m)\, \kappa_\alpha(\mu^n)
\end{align*}
\end{proof}

Because $Q_\mu:L^\infty(\Projm)\to L^\infty(\Projm)$  is a Markov operator,
\begin{proposition}\label{bounded}
$\Qop_\mu$  satisfies:
\begin{enumerate}
	\item  $(Q_\mu)^n=Q_{\mu^n}$,\, $\forall n\in\N$,
	\item $Q_\mu$ is bounded with  $\norm{Q_\mu}_\infty=1$.
\end{enumerate}
\end{proposition}

\begin{proposition}
Given $\hat{v},\hat{w}\in\Projm$, $\hat{v}\neq\hat{w}$,
$$\sup_{\hat{v}\neq\hat{w}}\int\left[\frac{d(\hat{g}\hat{v},\hat{g}\hat{w})}{d(\hat{v},\hat{w})}\right]^\alpha \, d\mu(g)\leq\sup_{\norm{v}=1}\int\left[\frac{\norm{\wedge^2g}}{\norm{gv}^2}\right]^\alpha \, d\mu(g) .$$
\end{proposition}

\begin{proof}
If $\norm{v}=\norm{w}=1$ then 
\begin{align*}
\left[\frac{d(gv,gw)}{d(v,w)}\right]^\alpha &= \left[\frac{\norm{gv\wedge gw}}{\norm{v\wedge w}\norm{gv}\norm{gw}}\right]^\alpha =\frac{\norm{\wedge^2g \frac{v\wedge w}{\norm{v\wedge w}}}^\alpha}{(\norm{gv}\norm{gw})^\alpha}\\
&\leq \left[\frac{\norm{\wedge^2 g}}{\norm{gv}\norm{gw}}\right]^\alpha
\leq\frac{1}{2}\left[\frac{\norm{\wedge^2 g}}{\norm{gv}^2}\right]^\alpha + \frac{1}{2}\left[\frac{\norm{\wedge^2 g}}{\norm{gw}^2}\right]^\alpha
\end{align*}
The result follows from taking the integral and the supremum in $\hat{v}\neq\hat{w}$ on both sides. We used Jensen's inequality and the fact that $ab\leq\frac{1}{2}(a^2+b^2)$ for $a,b\geq 0$.
\end{proof}

\begin{remark}
\label{rmk under Theta=kappa}
For $\mu\in \Prob(\SL_2(\R))$,\;  $\kappa_\alpha(\mu)=\underline\Theta_{2\alpha}(\mu)$.
\end{remark}

\begin{proof}
Since $\norm{\wedge_2 g}=1$ for $g\in \SL_2(\R)$, the previous inequality shows that $\kappa_\alpha(\mu)\leq \underline\Theta_{2\alpha}(\mu)$.
Since $\Pp(\R^2)$ is $1$-dimensional, given 
$\hat{v},\hat{w}\in\Projm$, $\hat{v}\neq\hat{w}$, 
there exists a unit vector $v_\ast $ in the convex cone generated by $v$ and $w$ such that
$$  \frac{d(\hat{g}\hat{v},\hat{g}\hat{w})}{d(\hat{v},\hat{w})}=\frac{1}{\norm{g\, v_\ast}^2} . $$
Integrating in $g$ w.r.t. $\mu$, the converse inequality follows.
\end{proof}

\begin{proposition}\label{cont2}
Given  $0<\alpha < \frac{p}{10}$, the function $\mu\mapsto \kappa_\alpha(\mu)$ is H\"older continuous with respect to $W_\alpha$.
\end{proposition}

\begin{proof}
 Fix $v\in\Projm$ and $T>0$ large and consider the set $\Bscr_T(\hat v)$ defined in~\eqref{BT(v) defin}.  We begin   finding the $\alpha$-H\"older constant of the following  map over the complement of this set
 $$\Bscr_T(\hat v)^\complement \ni g \mapsto  \frac{\norm{\wedge^2g}}{\norm{gv}^2} .$$ Notice that 

\begin{align*}  \left|\frac{\norm{\wedge^2g_1}}{\norm{g_1v}^2}-\frac{\norm{\wedge^2g_2}}{\norm{g_2v}^2}\right| 
\hspace{-2cm}& \hspace{2cm} \leq \left| \frac{\norm{\wedge^2g_1}\norm{g_2v}^2-\norm{\wedge^2g_2}\norm{g_1v}^2}{\norm{g_1v}^2\norm{g_2v}^2} \right|\\
& \leq  \left|\norm{\wedge^2g_1}-\norm{\wedge^2g_2}\right|\norm{g_1v}^{-2} + \norm{\wedge^2g_2} \left| \norm{g_1v}^{-2}-\norm{g_2v}^{-2} \right|  \\
&\leq e^{2 T} \bigg( \left|\norm{\wedge^2g_1}-\norm{\wedge^2g_2}\right| + \left| \norm{g_1v}^{-2}-\norm{g_2v}^{-2} \right| \bigg)\\
&\leq 3 e^{5 T}\, \norm{g_1-g_2}
\end{align*}
where in the last inequality we use  that (see~\cite[Lemma 2.8]{DK-book})
\begin{align*}
\left| \norm{\wedge^2g_1}-\norm{\wedge^2g_2}\right| &\leq \norm{\wedge^2g_1-\wedge^2g_2}=\norm{\wedge^2(g_1-g_2)}\\
&\leq 2 \max\{\norm{g_1},\norm{g_2}\}\norm{g_1-g_2}\leq  2 e^T\norm{g_1-g_2}
\end{align*}
and
\begin{align*}
\left| \norm{g_1v}^{-2}-\norm{g_2v}^{-2} \right| &= \left( \frac{1}{\norm{g_1 v} \norm{g_1 v}^2 } + \frac{1}{\norm{g_2 v} \norm{g_2 v}^2} \right)\,  \left| \norm{g_1 v}-\norm{g_1 v}\right| \\
&  \leq 2\,e^{3 T}\, \norm{g_1 v-g_2 v}  \leq 2\,e^{3 T}\, \norm{g_1 -g_2 }.
\end{align*}
This implies that the map
$$\Bscr_T(\hat v)^\complement \ni g \mapsto  \bigg(\frac{\norm{\wedge^2g}}{\norm{gv}^2}\bigg)^\alpha =:\psi_{\alpha, \hat v}(g) $$
is $\alpha$-H\"older with the constant
$3\, e^{5 \alpha T}$.

Next we use Lemma~\ref{lemma almost Holder continuity I} to derive the modulus of continuity for
$$\Mcal\ni \mu\mapsto \int \psi_{\alpha, \hat v}(g)\,  d\mu(g),$$
for any fixed $v\in\Projm$.
We  derive explicit estimates for the terms in the lemma's bound:

By Lemma~\ref{bound BT(v)}, we have
$$\max\{ \sqrt{\mu(\Bscr_T(\hat v))},\, \sqrt{\nu(\Bscr_T(\hat v))}\}\leq \sqrt{2}\, C^{1\over 2}\, e^{-\frac{p T}{2}}.$$
Since $2\alpha\leq p$ and $\mu, \nu\in \Mcal$, we have
$$ \max\{ \Theta_{2\alpha}(\mu),\, \Theta_{2\alpha}(\nu)\}\leq C. $$
Finally, since $8\alpha\leq p$,
\begin{align*}
	\norm{\psi_{\alpha, \hat v}}_{L^2(\mu)} &= \sqrt{\int  \bigg(\frac{\norm{ \wedge^2 g}^{2\alpha}}{\norm{gv}^{4\alpha}}\bigg) d\mu(g)}\leq \sqrt{\int  \bigg(\frac{\norm{g}^{4\alpha}}{\norm{gv}^{4\alpha}}\bigg) d\mu(g)}\\
	&\leq  \bigg(\int\norm{g}^{8\alpha} d\mu(g)\bigg)^{1/4}\bigg(\int\frac{1}{\norm{gv}^{8\alpha}} d\mu(g)\bigg)^{1/4} \leq   C^{1/2}    .
\end{align*}
All together, by Lemma~\ref{lemma almost Holder continuity I} we get
\begin{align*}
 \left| \int \bigg(\frac{\norm{\wedge^2g}}{\norm{gv}^2}\bigg)^\alpha d\mu(g) - \int \bigg(\frac{\norm{\wedge^2g}}{\norm{gv}^2}\bigg)^\alpha d\nu(g) \right|  \hspace{-5cm} &\\
   &\leq 3\, e^{5 \alpha T}\, W_\alpha(\mu,\nu)+ 24\, C\, e^{-(\frac{p}{2}-5\, \alpha)\,T}
\end{align*}

Next, choosing $T>0$ such that $$   W_\alpha(\mu,\nu) = 8\, C\, e^{-\frac{p}{2}\,T} $$
 we get $$e^T=\bigg(\frac{8\,C}{  W_\alpha(\mu,\nu)}\bigg)^{\frac{2}{p}},$$ which provides the explicit modulus of continuity
$$\left| \int \bigg(\frac{\norm{\wedge^2g}}{\norm{gv}^2}\bigg)^\alpha d\mu(g) - \int \bigg(\frac{\norm{\wedge^2g}}{\norm{gv}^2}\bigg)^\alpha d\nu(g) \right|\leq C_\ast W_\alpha (\mu,\nu)^{\frac{p-10\, \alpha}{p}},$$
where
$$ C_\ast =6\, \bigg( 8\, C \bigg)^{\frac{10\,\alpha}{p}} .$$
The H\"older continuity of $\mu\mapsto \kappa_\alpha(\mu)$ follows from the established equi-continuity in $\hat v\in \Projm$. More generally,  if $\{f_\lambda:X\to\R\}_{\lambda\in \Lambda}$ is  an equi-continuous family functions $f_\lambda:X\to \R$,
depending continuously on a parameter $\lambda$   ranging over a compact space $\Lambda$, then the pointwise  maximum function $f(x):=\max_{\lambda\in \Lambda} f_\lambda(x)$
 inherits the common modulus of continuity of the family.
%
%
\end{proof}

\begin{corollary}
	\label{coro cont2}
    The function $\mu\mapsto \kappa_\alpha(\mu^n)$ is continuous.
\end{corollary}
\begin{proof}
    This is a corollary of Proposition~\ref{cont1} and Proposition~\ref{cont2}.
\end{proof}

The next proposition is key to establish the strong mixing or uniform ergodicity of the Markov operator.

\begin{proposition}\label{8}
There exist $0<\alpha<1$, $\sigma_0 <1$ and $n\in\mathbb{N}$ such that $$\kappa_\alpha(\mu^n)\leq \sigma_0.$$
\end{proposition}

\begin{proof} 
By Lemma \ref{unif} we have   $$\lim_{n\to \infty} \int\frac{1}{n}\log\norm{gv}^{-2}d\mu^n(g)= -2L_1(\mu)$$ 
uniformly for every $\hat v\in\Projm$. 

Hence, given $\hat v\in \Projm$ and $\epsilon>0$, for all large enough $n$ 
$$\left| \int\frac{1}{n}\log\norm{gv}^{-2}\, d\mu^n(g) + 2L_1(\mu)\right|\leq\epsilon.$$ 
Choosing $\epsilon<\frac{L_1(\mu)-L_2(\mu)}{4}$ and  $n$ sufficiently large,  $$\int\log\norm{gv}^{-2} \ d\mu^n(g)\leq n(-2L_1(\mu)+\epsilon).$$
Denote the $i$-th singular value of $g\in \Mat_m(\R)$ by $s_i(g)$, $1\leq i\leq m$.
We also have that 
\begin{align*}
    \int\log\frac{\norm{\wedge^2 g}}{\norm{gv}^2} \ d\mu^n(g) &= \int\log\norm{\wedge^2 g} \ d\mu^n(g) + \int\log\norm{gv}^{-2} \ d\mu^n(g)\\
    &=\int \log\abs{s_1(g)} + \log\abs{s_2(g)} \ d\mu^n(g) + \int \log \norm{gv}^{-2} \ d\mu^n(g)\\
    &\leq n(L_1(\mu)+L_2(\mu)+\epsilon) + n(-2L_1(\mu)+\epsilon)\\
    &\leq  n(L_2(\mu)-L_1(\mu)+2\,\epsilon) \leq -1 ,
\end{align*}
the last step because  $L_1(\mu)-L_2(\mu)>4\epsilon>0$ and $n$ is large enough.

Now notice that $$e^{x} \leq 1 + x + \frac{x^2}{2}e^{\left|x\right|}.$$
Thus, for every $\hat v\in\Projm$
\begin{align*}
    \int\bigg({{\frac{\norm{\wedge^2 g}}{\norm{gv}^2}}}\bigg)^\alpha \ d\mu^n(g) &= \int e^{\alpha\log\frac{\norm{\wedge^2 g}}{\norm{gv}^2}} \ d\mu^n(g)\\
    &\leq \int 1 + \alpha\log{\frac{\norm{\wedge^2 g}}{\norm{gv}^2}} + \frac{{\bigg(\alpha\log{\frac{\norm{\wedge^2 g}}{\norm{gv}^2}}}\bigg)^2}{2}e^{\left| \alpha\log{\frac{\norm{\wedge^2 g}}{\norm{gv}^2}} \right|} \ d\mu^n(g)\\
    &\leq 1 - \alpha + O(\alpha^2)<1
\end{align*}
provided $\alpha>0$ is sufficiently small.
The bound $O(\alpha^2)$ holds for $\alpha<\frac{p}{8}$ by finiteness of the following moment
\begin{align*}
 \int \log^2{\frac{\norm{\wedge^2 g}}{\norm{gv}^2}} \; e^{\left| \alpha\log{\frac{\norm{\wedge^2 g}}{\norm{gv}^2}} \right| }\, d\mu^n(g) &=
  \int \log^2\bigg( {\frac{\norm{\wedge^2 g}}{\norm{gv}^2}} \bigg) \;  \bigg( \frac{\norm{\wedge^2 g}}{\norm{gv}^2} \bigg)^{\frac{p}{8}} \, d\mu^n(g)\\
  &\leq 
  \int  \bigg( \frac{\norm{\wedge^2 g}}{\norm{gv}^2} \bigg)^{\frac{p}{4}}  \, d\mu^n(g) <\infty.
\end{align*} 
See the proof of Proposition~\ref{cont2}.
\end{proof}

Finally we prove the  strong mixing property (also known as  uniform ergodicity) of the Markov operator.

\begin{theorem}[Uniform strong mixing]\label{sm}
Given $\mu\in\Mcal$, $0<\alpha<\frac{p}{8}$ as in Proposition~\ref{8}  there exist constants  $K>0$ and $0<\sigma<1$ and a neighborhood $\mathscr{V}\subset  \Mcal$ of $\mu$ such that every  $\nu\in \mathscr{V}$ 
  admits a unique stationary measure  $\eta_{\nu}\in \Prob(\Projm)$ satisfying
 $$\norm{Q^n_\nu(\varphi)-\int\varphi \ d\eta_{\nu}}_\alpha\leq K\sigma^n\norm{\varphi}_\alpha ,$$
 for all   $n\in\mathbb{N}$ and $\varphi\in\Holal(\Projm)$. 
\end{theorem}

\begin{proof}
	The existence of stationary measures follows from Proposition~\ref{space of stationary measures}, while uniqueness is a consequence of the theorem's conclusion. The proof proceeds by combining propositions~\ref{7},~\ref{spectral property},~\ref{bounded},~\ref{8} and Corollary~\ref{coro cont2} in a standard manner. Fix $\alpha>0$, $n_0\in \N$ and $\sigma_0<1$ according to Proposition \ref{8}. By Corollary~\ref{coro cont2}
	this measurement holds for all measures $\nu\in \mathscr{V}$ in a small neighborhood $\mathscr{V}$ of $\mu$. 
	Therefore, using propositions~\ref{7} and ~\ref{spectral property} we get	
	$$v_\alpha(\Qop_\nu^{n  n_0}\varphi)\leq \kappa_\alpha(\nu^{n_0})^n\, v_\alpha(\varphi) \leq \sigma_0^n\, v_\alpha(\varphi) $$
	and a simple division argument gives for some explicit constant $C_\ast<\infty$ and all $n\in \N$, 
	$$v_\alpha(\Qop_\nu^{n}\varphi)\leq  C_\ast\, \sigma^n\, v_\alpha(\varphi), $$
	with $\sigma=\sigma_0^{1\over n_0}<1$.	
	Finally, this implies that	
	$$ \norm{Q^n_\mu(\varphi)-\int\varphi \ d\eta}_\alpha\leq C_\ast\, \sigma^n\, v_\alpha(\varphi) .$$
\end{proof}

\section{H\"older continuity of the Lyapunov Exponent}\label{Holder}

In this section we prove Theorem~\ref{thm non invertible Holder LE }, i.e., the continuity of the first Lyapunov exponent
under the following hypothesis: (i) $\mu\in\Mcal$, (ii) $L_1(\mu)>L_2(\mu)$  and (iii)
$\mu$ is quasi-irreducible.
We will need the following lemmas.


\begin{lemma}
    Given $v,w\in\mathbb{R}^m\setminus\{0\}$, $$\norm{\frac{v}{\norm{v}}-\frac{w}{\norm{w}}}\leq\max\bigg\{\frac{1}{\norm{v}},\frac{1}{\norm{w}}\bigg\}\norm{v-w}.$$
\end{lemma}
\begin{proof}
    See Proposition 2.26 of \cite{DK-book}.
\end{proof}

\begin{lemma}\label{3lemma}
    Let $g_1,g_2\in \Mat_m(\R)$ and consider their projective actions $\hat{g_i}:\Projm\rightarrow \Projm$ defined as $\hat{v}\mapsto \frac{g_i(v)}{\norm{g_i(v)}},$ $i\in\{1,2\}$. Then $$d(\hat{g_1}\hat{v},\hat{g_2}\hat{v})\leq\norm{g_1-g_2}\max\bigg\{\frac{1}{\norm{g_1(v)}},\frac{1}{\norm{g_2(v)}}\bigg\}.$$
\end{lemma}
\begin{proof}
    See Proposition 2.9 of \cite{DK-book}.
\end{proof}

Given $\mu\in \Mcal$ under this section's hypothesis,
by propositions~\ref{spectral property} and~\ref{8},
the following series converges geometrically
$$ \Sigma_\alpha(\mu):=\sum_{j=0}^\infty \kappa_\alpha(\mu^j) <+\infty .$$
By Corollary~\ref{coro cont2}, the  sum $\Sigma_\alpha(\mu)$ is also locally bounded.

\begin{lemma}\label{mo}
    Given $\mu,\mu'\in\Mcal$, assume that
    $\Sigma_\alpha(\mu)<K$   for some $0<\alpha\leq \frac{p}{2}$. Then for all $n\in\mathbb{N}$ and $\phi\in\Holal(\Projm)$, $$\norm{Q^n_{\mu}(\phi)-Q^n_{\mu'}(\phi)}_\infty\leq K\, C^{{\alpha\over p}}\,  v_\alpha(\phi)\, W_p(\mu,\mu')^\frac{\alpha}{p}.$$ Moreover, if also $\kappa_\alpha((\mu')^n)<1$, $n\in \N$, then for all $\phi\in \Holal(\Projm)$ 
    $$\left|\int_{\Projm}\phi \ d\eta-\int_{\Projm}\phi \  d\eta'\right|\leq K\, C^{\alpha\over p}\,  v_\alpha(\phi)\, W_p(\mu,\mu')^\frac{\alpha}{p},$$ where $\eta$ and $\eta'$ are the unique stationary measures associated, respectively, to $\mu$ and $\mu'$.
\end{lemma}

\begin{proof}
For $n=1$ and for every $\pi\in\Pi(\mu,\mu')$
\begin{align*}
\norm{\Qop_{\mu}(\phi) - \Qop_{\mu'}(\phi)}_\infty \hspace{-2cm} &\hspace{2cm} =  \sup_{\hat{v}\in\Projm} \left| \int \phi(\hat{g_1} \hat{v}) \, d\mu(g_1)- \int \phi(\hat{g_2}\hat{v}) \, d\mu'(g_2)\right| \\
&\leq \sup_{\hat{v}\in\Projm} \int \left| \phi(\hat{g_1}\hat{v}) - \phi(\hat{g_2}\hat{v}) \right| \, d\pi(g_1, g_2)\\
&\leq \sup_{\hat{v}\in\Projm} \int \frac{ \left| \phi(\hat{g_1}\hat{v}) - \phi(\hat{g_2}\hat{v}) \right| }{d(\hat{g_1}\hat{v},\hat{g_2}\hat{v})^\alpha}\,  d(\hat{g_1}\hat{v},\hat{g_2}\hat{v})^\alpha  \, d\pi(g_1, g_2)\\
&\leq v_\alpha(\phi) \sup_{\hat{v}\in\Projm} \int d(\hat{g_1}\hat{v},\hat{g_2}\hat{v})^\alpha \, d\pi(g_1, g_2)\\
&\leq v_\alpha(\phi) \sup_{\hat{v}\in\Projm} \int \|g_1 - g_2\|^{\alpha}\cdot \max \left\{ \frac{1}{\|g_1 v\|}, \frac{1}{\|g_2 v\|} \right\}^\alpha d\pi(g_1, g_2),
\end{align*}
where the last inequality is due to Lemma \ref{3lemma}. 
Hence, by H\"older and Jensen inequalities we get
\begin{align*}
\hspace{2cm}& \hspace{-2cm} \sup_{\hat{v}\in\Projm} \int \|g_1 - g_2\|^{\alpha}\cdot \max \left\{ \frac{1}{\|g_1 v \|}, \frac{1}{\|g_2 v \|} \right\}^\alpha d\pi(g_1, g_2)\\
&\leq  C^{\alpha/p}\,   \bigg(\int\norm{g_1-g_2}^{p} \ d\pi(g_1,g_2)\bigg)^{{\alpha\over p}}\\
&\leq  C^{\alpha/p}\, W_p(\mu,\mu')^{\alpha\over p} .
\end{align*}
Now notice that the difference $\Qop_{\mu}^{n} - \Qop_{\mu'}^{n}$ can be written as a telescopic sum in the following manner:
\begin{align*}
\Qop_{\mu}^{n} - \Qop_{\mu'}^{n}=\sum_{i=0}^{n-1} \Qop_{\mu'}^{i} \circ (\Qop_{\mu} - \Qop_{\mu'})\circ Q^{n-i-1}_{\mu}.
\end{align*}
Hence, because $\Qop_\mu$ is Markov (see Proposition~\ref{bounded}),
\begin{align*}
    \norm{Q^n_{\mu}(\phi)-Q^n_{\mu'}(\phi)}_\infty&=\norm{\sum_{i=0}^{n-1} \Qop_{\mu'}^{i} \circ (\Qop_{\mu} - \Qop_{\mu'})\circ Q^{n-i-1}_{\mu}(\phi)}_\infty\\
    &=\sum_{i=0}^{n-1}\norm{\Qop_{\mu'}^{i} \circ (\Qop_{\mu} - \Qop_{\mu'})\circ Q^{n-i-1}_{\mu}(\phi)}_\infty\\
    &\leq\sum_{i=0}^{n-1}\norm{(\Qop_{\mu} - \Qop_{\mu'})\circ Q^{n-i-1}_{\mu}(\phi)}_\infty.\\
    &\leq C^{\alpha\over p} \, \sum_{i=0}^{n-1} v_\alpha(Q^{n-i-1}_\mu(\phi)) W_p(\mu,\mu')^{\alpha\over p}\\
    &\leq C^{\alpha\over p}  \, W_p(\mu,\mu')^{\alpha/p} \,\sum_{i=0}^{n-1}\kappa_\alpha(\mu^{n-i-1})\, v_\alpha(\phi)\\
    &\leq   K\,C^{\alpha\over p}  \,  
    v_\alpha(\phi)\, W_p(\mu,\mu')^{\alpha\over p} 
\end{align*}
Now, if $\kappa_\alpha((\mu')^n)<1$,  the uniform convergence in Theorem \ref{sm} applies to $\mu'$ as well and $$\left|\int \phi \ d\eta - \int \phi \ d\eta'\right|\leq\sup_{n\in\mathbb{N}}\norm{Q^n_{\mu}(\phi)-Q^n_{\mu'}(\phi)}_\infty\leq K\,C^{\alpha\over p}  \, v_\alpha(\phi)\, W_p(\mu,\mu')^{\alpha\over p}.$$
\end{proof}

\begin{proof}[Proof of Theorem~\ref{thm non invertible Holder LE }]
Let $\mathcal{U} \subset \Mcal$ denote the subset of measures $\mu \in \Mcal$ satisfying the assumptions of the theorem, namely that $\mu$ is quasi-irreducible and $L_1(\mu) > L_2(\mu)$. 

Given $\hat{v} \in \mathbb{P}^{m-1}$, define the function $\psi_{\hat{v}} : \Mat_m(\R) \to \R$ by
\[
\psi_{\hat{v}}(g) := \log \|g v\|.
\]

Let $\mu, \mu' \in \mathcal{U}$, and denote by $\eta$ and $\eta'$ their respective (unique) stationary measures. By Proposition~\ref{8}, there exist $n \in \N$, $0 < \alpha \leq \frac{p}{2}$, and $0 < \sigma_0 < 1$ such that $\kappa_\alpha((\mu')^n) \leq \sigma_0$ for all measures $\mu'$ sufficiently close to $\mu$. 

By Furstenberg's formula (see Lemma~\ref{lemma all beta equal}), we have:
\begin{align*}
	|L_1(\mu) - L_1(\mu')| 
	&= \left| \iint \psi_{\hat{v}}(g) \, d\mu(g) \, d\eta(\hat{v}) - \iint \psi_{\hat{v}}(g) \, d\mu'(g) \, d\eta'(\hat{v}) \right| \\
	&\leq \left| \iint \psi_{\hat{v}}(g) \, d\mu(g) \, d\eta(\hat{v}) - \iint \psi_{\hat{v}}(g) \, d\mu(g) \, d\eta'(\hat{v}) \right| \\
	&\quad + \left| \iint \psi_{\hat{v}}(g) \, d\mu(g) \, d\eta'(\hat{v}) - \iint \psi_{\hat{v}}(g) \, d\mu'(g) \, d\eta'(\hat{v}) \right|.
\end{align*}

We begin by estimating the second term:
\[
\left| \iint \psi_{\hat{v}} \, d\mu \, d\eta'(\hat{v}) - \iint \psi_{\hat{v}} \, d\mu' \, d\eta'(\hat{v}) \right| 
\leq \int \left| \int \psi_{\hat{v}} \, d\mu - \int \psi_{\hat{v}} \, d\mu' \right| \, d\eta'(\hat{v}).
\]

The function $\psi_{\hat{v}}$ is continuous except where $\psi_{\hat{v}} = -\infty$. Note that the set $\{ |\psi_{\hat{v}}| > T \}$ is contained in the small set $\Bscr_T(\hat{v})$ defined in~\eqref{BT(v) defin}. 

Moreover, for any $\hat{v} \in \Pp^{m-1}$, the function $g \mapsto \psi_{\hat{v}}(g)$ is $p$-H\"older continuous on $\{|\psi_{\hat v}|\leq T\}$. Indeed, if $|\psi_{\hat v}(g_1)|\leq T$ and   $|\psi_{\hat v}(g_2)|\leq T$,
\begin{align*}
	\left| \log \|g_1 v\| - \log \|g_2 v\| \right| 
	&\leq \frac{1}{p} \left| \log(\|g_1 v\|^p) - \log(\|g_2 v\|^p) \right| \\
	&\leq \frac{1}{p} e^{p T} \|g_1 v - g_2 v\|^p 
	\leq \frac{1}{p} e^{p T} \|g_1 - g_2\|^p.
\end{align*}

Applying Lemma~\ref{lemma almost Holder continuity II} to $\psi_{\hat{v}}$ with $L = \frac{1}{p} e^{p T}$, we obtain:
\begin{align*}
	\left| \iint \psi_{\hat v}\, d\mu\, d\eta'(\hat{v})- \iint \psi_{\hat v} \ d\mu'\, d\eta'(\hat{v}) \right| \hspace{-5cm} &\\
	&\leq \frac{1}{p}\, e^{p T}\, W_p(\mu,\mu') + 4\, C\, T\,  e^{-p\, T}  + 2\, \sqrt{2}\, C \, e^{-{{p T}\over 2}}\\
	&\leq \frac{1}{p}\, \left[\, e^{p\, T}\, W_p(\mu,\mu') + 8\, C\, e^{-\frac{p\, T}{2}} \right] \leq   2\, \frac{ (8\, C)^{2\over 3}}{p} \, W_p(\mu,\mu')^{\frac{1}{3}} ,
\end{align*}
provided $T$ is large enough so that $T e^{-p T} < e^{-p T / 2}$. This final bound is obtained by choosing $T$ such that  $e^{p\, T}\, W_p(\mu,\mu')= 8\, C\, e^{-\frac{p\, T}{2}}$ , which implies that $e^{p\, T}=\bigg(\frac{8\, C}{W_p(\mu,\mu')}\bigg)^\frac{2}{3}$.

Next, we estimate the first term in the bound for $|L_1(\mu) - L_1(\mu')|$. Let $\phi(\hat{v}) := \int \psi_{\hat{v}}(g) \, d\mu(g)$. Then, by Theorem~\ref{sm} and Lemma~\ref{mo},
\begin{align*}
	\left| \iint \psi_{\hat{v}} \, d\mu \, d\eta - \iint \psi_{\hat{v}} \, d\mu \, d\eta' \right|
	&= \left| \int \phi \, d\eta - \int \phi \, d\eta' \right| \\
	&\leq \sup_{n \in \mathbb{N}} \| Q_\mu^n(\phi) - Q_{\mu'}^n(\phi) \|_\infty \\
	&\leq K \, C^{\alpha/p} \, v_\alpha(\phi) \, W_p(\mu, \mu')^{\alpha/p},
\end{align*}
for some  $K = K(\mu) < \infty$ that is uniform in a neighborhood of $\mu$. 

To complete the proof, it remains to show that $v_\alpha(\phi) < \infty$. Observe:
\begin{align*}
	\left| \log \|g v\|^\alpha - \log \|g v'\|^\alpha \right| 
	&\leq \max \left\{ \frac{1}{\|g v\|^\alpha}, \frac{1}{\|g v'\|^\alpha} \right\} \left| \|g v\|^\alpha - \|g v'\|^\alpha \right|, \\
	\left| \|g v\|^\alpha - \|g v'\|^\alpha \right| 
	&\leq \left| \|g v\| - \|g v'\| \right|^\alpha 
	\lesssim \|g\|^\alpha \, d(v, v')^\alpha.
\end{align*}

Therefore,
\begin{align*}
	|\phi(v) - \phi(v')|
	&\leq \frac{1}{\alpha} \int \left| \log \|g v\|^\alpha - \log \|g v'\|^\alpha \right| d\mu(g) \\
	&\lesssim \frac{1}{\alpha} \int \max \left\{ \frac{1}{\|g v\|^\alpha}, \frac{1}{\|g v'\|^\alpha} \right\} \|g\|^\alpha \, d\mu(g) \, d(v, v')^\alpha \\
	&\lesssim \frac{d(v, v')^\alpha}{\alpha} \left( \int \max \left\{ \frac{1}{\|g v\|^{2\alpha}}, \frac{1}{\|g v'\|^{2\alpha}} \right\} d\mu(g) \right)^{1/2} \left( \int \|g\|^{2\alpha} d\mu(g) \right)^{1/2} \\
	&\lesssim \frac{1}{\alpha} \, d(v, v')^\alpha,
\end{align*}
for any $0 < \alpha < \frac{p}{2}$, completing the proof.
\end{proof}

\begin{proof}[Proof of Theorem~\ref{thm L1 Holder LE}]
Given $A\in \mathscr{C}_p(\Omega,\SL_m(\R))$,
	the $p$-exponential moment of the measure $\nu:=A_\ast \Pp_\mu$ in~\eqref{Theta p} coincides  with the corresponding  moment
	$\Theta_p(A)$ of the cocycle $A$ defined in~\eqref{Theta p (A)}. Since
	\begin{align*}
	\abs{\Theta_p(A)-\Theta_p(B)} &\leq \int \abs{\norm{A}^p-\norm{B}^p}\, d\Pp_\mu  \leq \int \norm{A-B}^p\, d\Pp_\mu \\
	&\leq \bigg(  \int \norm{A-B}\, d\Pp_\mu \bigg)^p =d(A,B)^p,
	\end{align*}
	we see that  $p$-exponential moments are locally bounded in  
$\mathscr{C}_p(\Omega, \SL_m(\R))$.
Given $A, B\in \mathscr{C}_p(\Omega, \SL_m(\R))$, making use of  the coupling $\pi=\smallint \delta_{(A,B)}\, d\Pp_\mu$ we get
\begin{align*}
W_\alpha (A_\ast\Pp_\mu, B_\ast\Pp_\mu) &\leq
\int \norm{A-B}^\alpha\, d\Pp_\mu \leq \bigg(  \int \norm{A-B}\, d\Pp_\mu \bigg)^p =d(A,B)^\alpha .
\end{align*}
Therefore, under the theorem's hypothesis, since by Theorem~\ref{thm non invertible Holder LE }
the Lyapunov exponent $L_1(A)=L_1(A_\ast\Pp_\mu)$ depends H\"older continuously on the measure $\nu=A_\ast\Pp_\mu$ w.r.t. the Wasserstein distance,
the above inequality shows that  it also depends H\"older continuously  on $A$ w.r.t. the $L^1$ distance.	
\end{proof} 

\section{Large deviations}\label{Large}

This section is devoted to proving Theorem~\ref{thm large deviation estimates}, which follows from an abstract large deviation result~\cite[Theorem 2.1]{CDK-paper3}.

Let $K:\Pp(\R^m)\to \Prob(\Pp(\R^m))$ denote the kernel defined by
\[
K_{\hat v} := \int \delta_{\hat g \hat v} \, d\mu(g),
\]
which induces the Markov operator
\[
\Qop_\mu: C^0(\Pp(\R^m)) \to C^0(\Pp(\R^m)), \quad (\Qop_\mu f)(\hat v) := \int f(\hat g \hat v) \, d\mu(g).
\]
Additionally, consider the continuous function $\Psi: \Pp(\R^m) \to \R$,
\[
\Psi(\hat v) := \int \log \|g\, v\| \, d\mu(g).
\]

By Theorem~\ref{sm}, under the assumptions of Theorem~\ref{thm large deviation estimates} (which coincide with those of Theorem~\ref{thm non invertible Holder LE }), there exist a unique invariant measure $\eta \in \Prob_\mu(\Pp(\R^m))$ and constants $C < \infty$ and $0 < \sigma < 1$ such that, for any $\varphi \in \Holal(M)$ and $n \in \mathbb{N}$,
\[
\left\| \Qop_\mu^n \varphi - \int \varphi \, d\eta \right\|_\infty \leq C \sigma^n \|\varphi\|_\alpha.
\]
Moreover, this strong mixing (uniform ergodicity) property is stable under perturbations of the measure $\mu$ within the space $\Mcal$.

Now, define the space $M := \Mat_m(\R) \times \Pp(\R^m)$ and consider the discontinuous kernel
\[
k: M \to \Prob(M), \quad k_{(g,\hat v)} := \mu \times \delta_{\hat g \hat v},
\]
along with the Markov operator
\[
\Qop_k: L^\infty(M) \to L^\infty(M), \quad (\Qop_k \varphi)(g, \hat v) := \int \varphi(g', g \hat v)\, d\mu(g'),
\]
and the discontinuous observable
\[
\psi: M \to \R, \quad \psi(g, \hat v) := \log \|g\, v\|.
\]
These objects are analyzed in detail in Section~\ref{FK}.

A straightforward computation shows that
\[
(\Qop_k^n \psi)(g_0, \hat v_0) = (\Qop_\mu^{n-1} \Psi)(g_0 \hat v_0),
\]
and therefore,
\[
\left\| \Qop_k^n \psi - L_1(\mu) \right\|_\infty \leq C \sigma^n \, \Lip(\Psi),
\]
where this exponential convergence is stable under perturbations of the generating measure $\mu \in \Mcal$.

By applying Theorem 2.1 and Remark 2.4 from~\cite{CDK-paper3} to the $k$-Markov process $\{\xi_n : \Omega \to M\}_{n \geq 0}$ introduced in Section~\ref{FK}, we obtain the following large deviation estimate: for all $n \geq 1$, $\varepsilon > 0$, and $\hat v \in \Pp(\R^m)$,
\begin{align*}
	& \mu^\N\left\{\, \omega\in \Mat_m(\R)^\N\colon \, \left\vert
	\frac{1}{n}\, \log \norm{A^n(\omega)\, v} - L_1(\mu)  \right\vert >\varepsilon  \, \right\} = \\
	&	\Pp_{\hat v} \left[ \,  \omega\in \Omega\, \colon\,  \left\vert
	\frac{1}{n}\, \sum_{j=0}^{n-1} \psi(\xi_j(\omega, \hat v)) - L_1(\mu)  \right\vert >\varepsilon  \, \right]  
	<C\, e^{-c(\varepsilon)\, n} ,
\end{align*}
where $\Pp_{\hat v}$ denotes the conditional probability given $\xi_0 = \hat v$.

By construction, this large deviation bound is robust under perturbations of the measure $\mu \in \Mcal$.

\section{Applications}\label{Schro}
In this section we give a few examples of applications of the main results to Schr\"odinger cocycles with unbounded potentials.

\subsection{Example 1} Let $\Omega=\R^\Z$ be the space of real number sequences, consider the function $\mathscr{V}:\Omega\to\R$ defined as
$\mathscr{V}(\omega):=\omega_0$ and the family of Schr\"odinger cocycles $A_E:\Omega\to \SL_2(\R)$ given by
$$  A_E(\omega):=\begin{bmatrix}
	\mathscr{V}(\omega) - E & -1\\ 1 & 0 
\end{bmatrix} .
$$

Given constants $C<\infty$ and $0<p\leq 1$, consider the space $\mathscr{L}^p_C$ (non-compactly supported) measures $\mu\in \Prob(\R)$ satisfying
$$ \int_\R |x|^p\, d\mu(x)\leq C .$$

Each choice of  $\mu\in \mathscr{L}^p_C$ determines a random Schr\"odinger cocycle over the full shift $\sigma:\Omega\to\Omega$ endowed with the Bernoulli measure
$\Pp_\mu:=\mu^\Z$.
Let $L_1(\mu,E)$  be the first Lyapunov exponent of this cocycle.

\begin{proposition} The  function  $\mathscr{L}^p_C\times \R\ni (\mu, E)\mapsto L_1(\mu, E)$ is  locally H\"older continuous with respect to the Wasserstein metric on $\mathscr{L}^p_C$.
\end{proposition}

\begin{proof}
The H\"older continuity of the Lyapunov exponent in $E$ follows from~\cite{LP89}. The general case follows from  Theorem~\ref{thm non invertible Holder LE } because these cocycles are always strongly irreducible and non-compact. Also because $L_1(\mu,E)>0$,
 by Furstenberg's positivity criterion.
\end{proof}

\bigskip

\subsection{Example 2}
Denote by $\Fscr^d_C$ the space of compactly supported measures
 $\mu\in \Prob_c(\R)$ such that for some $C<\infty$,  
 $$ \mu(B(x,r))\leq C\, r^d\quad \forall \, x\in \R, \; \forall \, r>0.$$
These  will be referred  to  as Frostman measures.

\begin{lemma}
\label{lemma Frostman measures}
Let $0<d\leq 1$. Given $\mu\in \Fscr^d_C$ and any $0<p<d$ there exists a constant $C'=C(d,C)<\infty$ such that
$$ \sup_{a\in \R} \, \int_{-\infty}^\infty \frac{1}{|t-a|^d}\, d\mu(t)\leq C' .$$
Conversely,  if  
$$ \sup_{a\in \R} \, \int_{-\infty}^\infty \frac{1}{\abs{t-a}^d}\, d\mu(t) \leq C.$$
then $\mu$ is a Frostman measure, with $\mu\in \Fscr^d_C$.
\end{lemma}

\begin{proof}
These facts are easy exercises. For the sake of completeness we include their proofs here. Consider the partition 
$$\Sigma_0:=B(a,1)^\complement\; \text{ and }\; \Sigma_n:= B(a,2^{-(n-1)})\setminus B(a,2^{-n}) \, \text{ for }\, n\geq 1.$$ 
Then $\smallint_{\Sigma_0} |t-a|^{-p}\, d\mu(t)\leq 1$ and for every $n\geq 1$,
\begin{align*}
	\int_{\Sigma_n} |t-a|^{-p}\, d\mu(t) &\leq  2^{n\, p}\, \mu(B(a,2^{-(n-1)}))\\
	&\leq
	2^{n\, p}\, C\, 2^{-(n-1)\, d} = 2^{d}\, C\,  2^{-(d-p)\, n} .
\end{align*}  
Adding up we get
\begin{align*}
	\int_{-\infty}^\infty \frac{1}{|t-a|^d}\, d\mu(t) &\leq 1
	+\sum_{n=1}^\infty 2^{d}\, C\,  2^{-(d-p)\, n} =1+ \frac{2^d\, C}{2^{d-p}-1} =:  C' ,
\end{align*} 
which proves the lemma's first statement.

Take $a\in \R$ and $r>0$ and denote by $\textbf{1}_{B(a,r)}$ the indicator function of the ball
$B(a,r):=\{x\in \R\colon \abs{x-a}\leq r\}$. Then 
$r^{-d}\, \textbf{1}_{B(a,r)}(t)\leq \abs{t-a}^{-d}$, which implies that
$$ r^{-d}\, \mu(B(a,r))\leq  \int_{-\infty}^\infty  \abs{t-a}^{-d}\, d\mu(t) \leq C .$$
This proves that $\mu$ is a Frostman measure.
\end{proof}

Given a Frostman measure  $\mu\in\Fscr^d_C$, $a\in \R$  
and $0<q<1$, consider the unbounded family of random potentials $\V_\lambda$    taking the value 
$\V_\lambda(\omega)=\omega_0$ with probability distribution
$$ (1-q)\, \delta_{\omega_0}\, d\mu(\omega_0) + q\, d\delta_a( \omega_0 \lambda^{-1}) , $$
where $\lambda$ is a large coupling constant.

The associated Schr\"odinger cocycle is defined as follows.
 Let $\Omega=(\R\times \{0,1\})^\Z$. Consider the Bernoulli measure 
 $\Pp_{\mu,q}:=((1-q)\, (\mu\times \delta_0) +q\,   \delta_{(a,1)})^\Z$ and the families of functions $\Lambda_\lambda, \mathscr{V}:\Omega\to\R$ defined by
$$ \Lambda_{\lambda}(\omega):=\begin{cases}
	1 &\text{ if  } \omega_0=(t_0,0)\\
	\lambda &\text{ if  } \omega_0=(t_0,1) 
\end{cases} \qquad \mathscr{V}(\omega):= t_0 \; \text{ where } \omega_0=(t_0, \cdot)  . $$
This data determines the Schr\"odinger cocycle  $A_{\lambda,E}:\Omega\to \SL_2(\R)$,
$$  A_{\lambda,E}(\omega):=\begin{bmatrix}
	\Lambda_\lambda(\omega)\, (\mathscr{V}(\omega)-E)& -1\\ 1 & 0 
\end{bmatrix} .
$$
Consider also the  normalized cocycle,
$$  \tilde A_{\lambda,E}(\omega):=\Lambda_\lambda(\omega)^{-1} A_{\lambda,E} = \begin{bmatrix}
	\mathscr{V}(\omega)-E & -\Lambda_\lambda(\omega)^{-1}\\ \Lambda_\lambda(\omega)^{-1} & 0 
\end{bmatrix} .
$$
The second cocycle  projects to the measure  $\mu_{\lambda^{-1},E,q}\in \Prob(\Mat_2(\R))$,
where for $\beta=\lambda^{-1}$,
$$  \mu_{\beta,E,q}:= (1-q)\,  \int \delta_{\begin{bmatrix}
		t-E  & -1 \\ 1 & 0
\end{bmatrix}} \, d\mu(t) \; + \; q\,  \delta_{\begin{bmatrix}
		a-E & -\beta \\ \beta  & 0
\end{bmatrix}} , $$
which converges in the Wasserstein metric, as $\lambda\to\infty$, to the measure 
$$   \mu_{0,E,q}:= (1-q)\,  \int \delta_{\begin{bmatrix}
		t-E  & -1 \\ 1 & 0
\end{bmatrix}} \, d\mu(t) \;  +\;  q\,  \delta_{\begin{bmatrix}
		a-E & 0 \\ 0  & 0
\end{bmatrix}} . $$

The map  induced  by the one sided shift $\sigma:\Omega\to\Omega$
on the cylinder 
$$ \mathscr{C}:=\left\{\omega\in\Omega\colon  \omega_0=(\cdot,0)\, \right\} . $$
can be viewed as a Bernoulli shift  $\sigma:\Mat_2(\R)^\Z\to\Mat_2(\R)^\Z$  endowed with the measure
\begin{equation}
	\label{eqdef tilde mu la E}
\tilde \mu_{\beta,E,q}:= \sum_{n=1}^\infty \sum_{m=1}^\infty q^m\, (1-q)^n\, \begin{bmatrix}
	a-E & -\beta \\ \beta  & 0
\end{bmatrix}^{m}_\ast  (\tilde \theta_{\mu,E})^{\ast n} ,	
\end{equation}  
where
\begin{align*}
\tilde \theta_{\mu, E} &:=  \int \delta_{\begin{bmatrix}
		t-E  & -1  \\ 1 & 0
\end{bmatrix}}\, d\mu(t) . 
\end{align*}

\begin{lemma}
	\label{Lemma Frostman theta}
There exist $0<p<d$ and $C'=C'(\mu, E)<\infty$ such that
$$   \sum_{n=1}^\infty 
\underline \Theta_p((\tilde \theta_{\mu,E})^{\ast n} ) \leq C' .$$
\end{lemma}

\begin{proof}
The probability measure $\tilde\theta_{\mu,E}$ on $\mathrm{SL}_2(\mathbb{R})$ has compact support. As it arises from a Schr\"odinger cocycle, it satisfies the assumptions underlying Proposition~\ref{8}. By applying this proposition along with Remark~\ref{rmk under Theta=kappa}, we conclude that there exist constants $0 < p < 1$, $0 < \sigma < 1$, and an integer $n_0 \in \mathbb{N}$ such that
\[
\underline{\Theta}_{2p}\big((\tilde\theta_{\mu,E})^{\ast n_0}\big)
= \kappa_{p}\big((\tilde\theta_{\mu,E})^{\ast n_0}\big) < \sigma < 1.
\]
The second statement then follows directly from Proposition~\ref{spectral property}.
\end{proof}

\begin{lemma}
	\label{exp2 moment bounds}
Given $E \in \mathbb{R} \setminus \{a\}$, there exists $0 < p < \frac{d}{2}$ such that the measures $\tilde\mu_{\beta,E,q}$ defined in~\eqref{eqdef tilde mu la E} have moments $\bar{\Theta}_p(\tilde\mu_{\beta,E,q})$ and $\underline{\Theta}_p(\tilde\mu_{\beta,E,q})$ uniformly bounded for all $\beta < 1$.
\end{lemma}

\begin{proof}
Assume $L$ is large enough so that $\abs{t-E}<L$ for all $t\in \supp(\mu)$.
All matrices in the support of $\tilde \theta_{\mu, E}$
have norm $\leq  L+1$. Hence $\bar \Theta_p( \tilde \theta_{\mu, E})\leq  L+1$  for any $0<p\leq 1$.

Given $C < \infty$, consider the compact convex set of projective Frostman measures (of dimension $\geq 2p$):
\[
\Escr^p_C := \left\{ \eta \in \Prob(\Proj) \colon \forall\, \hat e \in \Proj,\; \int \frac{1}{d(\hat x, \hat e)^{2p}}\, d\eta(\hat x) \leq C \right\}.
\]
Since the mapping
\[
t \mapsto \begin{bmatrix}
	t - E & -1 \\ 1 & 0 
\end{bmatrix} \cdot \hat v \in \Proj
\]
is a local diffeomorphism with derivatives uniformly bounded from above and below, and because $\mu \in \Fscr^d_C$ with $2p < d$, there exists a constant $\tilde C < \infty$ such that for all $\hat v \in \Proj$, we have
\[
\tilde \theta_{\mu,E} \ast \delta_{\hat v} \in \Escr^p_{\tilde C}.
\]
It follows that the operator
\[
\Qop_\mu^\ast \eta := \tilde\theta_{\mu,1,E} \ast \eta = \int \tilde \theta_{\mu,1,E} \ast \delta_{\hat v} \, d\eta(\hat v)
\]
preserves the space $\Escr^p_{\tilde C}$, since it maps $\eta$ to a convex combination of measures in $\Escr^p_{\tilde C}$. Thus, by iteration, we have
\[
(\tilde \theta_{\mu,E})^{\ast n} \ast \delta_{\hat v} \in \Escr^p_{\tilde C}, \quad \text{for all } n \in \mathbb{N}.
\]

Next, we estimate the moment $\underline{\Theta}_p\big((g_0^m)_\ast (\tilde\theta_{\mu,E})^{\ast n}\big)$, where $g_0$ denotes the matrix
\[
g_0 = \begin{bmatrix}
	a - E & -\beta \\
	\beta & 0 
\end{bmatrix}.
\]
Let $\hat e_u$ and $\hat e_s \in \Proj$ denote the eigendirections of $g_0$ corresponding, respectively, to its largest and smallest eigenvalues. For any unit vector $v \in \mathbb{R}^2$, the following elementary inequality holds:
\[
\|g_0^m v\| \geq \|g_0^m\| \cdot \cos \measuredangle(v, e_u) = \|g_0^m\| \cdot d(\hat v, \hat e_s).
\]
Applying Cauchy--Schwarz's inequality, together with the estimate for $\underline{\Theta}_p\big((\tilde\theta_{\mu,E})^{\ast n}\big)$ obtained in Lemma~\ref{Lemma Frostman theta}, and using the remark above that
\[
(\tilde\theta_{\mu,E})^{\ast n} \ast \delta_{\hat v} \in \Escr^p_{\tilde C},
\]
we proceed with the moment estimate.
\begin{align*}
	\int \frac{1}{\norm{g\, v}^p}\, d(g_0^m)_\ast (\tilde\theta_{\mu,E})^{\ast n}(g) \hspace{-3cm}&\hspace{3cm}=
		\int \frac{1}{\norm{g_0^m\, g\, v}^p}\, d  (\tilde\theta_{\mu,E})^{\ast n}(g) \\
		&= 	\int \frac{1}{\norm{g\, v}^p}\,  \frac{1}{\norm{g_0^m\, \frac{g\, v}{\norm{g\, v}}}^p}\, d (\tilde\theta_{\mu,E})^{\ast n}(g) \\
		&\leq \sqrt{ \underline\Theta_{2p}((\tilde\theta_{\mu,E})^{\ast n}) } \, \left( 
		\int  \frac{1}{\norm{g_0^m\, \frac{g\, v}{\norm{g\, v}}}^{2p}}\, d (\tilde\theta_{\mu,E})^{\ast n}(g) \right)^{1\over 2}\\	
		&= \sqrt{ \underline\Theta_{2p}((\tilde\theta_{\mu,E})^{\ast n}) } \, \left( 
\int  \frac{1}{\norm{g_0^m\, w}^{2p}}\, d( (\tilde\theta_{\mu,E})^{\ast n}\ast \delta_{\hat v})(\hat w) \right)^{1\over 2}\\		
		&\lesssim   \left( 
\int  \frac{1}{\norm{g_0^m}^{2p}\, d(\hat w, \hat e_s)^{2p}}\, d( (\tilde\theta_{\mu,E})^{\ast n}\ast \delta_{\hat v})(\hat w) \right)^{1\over 2}\\
&\lesssim  	 \frac{1}{\abs{a-E}^{p \, m}}\, \left( 
\int  \frac{1}{ d(\hat w, \hat e_s)^{2p}}\, d( (\tilde\theta_{\mu,E})^{\ast n}\ast \delta_{\hat v})(\hat w) \right)^{1\over 2} \leq \frac{{\tilde C}^{1\over 2}}{\abs{a-E}^{p \, m}}
\end{align*}
Therefore,
\[
\underline{\Theta}_p\big((g_0^m)_\ast (\tilde\theta_{\mu,E})^{\ast n}\big) \lesssim |a - E|^{-p m}.
\]
Summing over $n$ and $m$, we obtain
\[
\underline{\Theta}_p(\tilde\mu_{\beta,E,q}) \lesssim \sum_{n=1}^\infty \sum_{m=1}^\infty \frac{q^m}{|a - E|^{p m}} \, (1 - q)^n,
\]
which is bounded provided we choose $0 < p < d/2$ sufficiently small so that $q < |a - E|^p$.

Note that this bound is independent of $\beta$.
\end{proof}

As an application of Theorem~\ref{thm non invertible Holder LE }, we derive an asymptotic formula for the Lyapunov exponent 
$L_1(A_{\lambda, E}, \Pp_{\mu,q})=L_1(\mu_{\lambda, E, q})$.

\begin{proposition}
Given  $E\neq a$, there exists $0<p<{d\over 2}$ such that
$$ L_1(\mu_{\lambda, E, q}) =q\, \log \lambda +   L_1(\tilde \mu_{0, E,q}) + \mathscr{O}\left(\lambda^{-p} \right) \quad \text{as }\quad \lambda\to \infty .$$
\end{proposition}

\begin{proof}
Choose $0 < p < \frac{d}{2}$ according to Lemma~\ref{exp2 moment bounds}, and consider the family of measures $\tilde{\mu}_{\beta, E, q} \in \Mcal$ for some constant $C < \infty$. By Proposition~\ref{prop bound L1 mu}, we have
\[
L_1(\tilde{\mu}_{0, E, q}) > -\infty = L_2(\tilde{\mu}_{0, E, q}).
\]
The measure $\tilde{\mu}_{0, E, q}$ is quasi-irreducible since the horizontal axis is the only invariant line, and the top Lyapunov exponent is attained along this direction.

Moreover, the top Lyapunov exponent of the original cocycle is related to that of the induced cocycle by the formula (see~\cite[Proposition 4.18 and Exercise 4.8]{Viana2014}):
\[
L_1(\mu_{\lambda, E, q}) = q \log \lambda + L_1(\tilde{\mu}_{\lambda^{-1}, E, q}).
\]

Consider the matrix
\[
g_0(\beta) := 
\begin{bmatrix}
	a - E & -\beta \\ \beta & 0
\end{bmatrix},
\]
and let $\rho(\beta)$ denote its spectral radius. Using the inequality
\[
\abs{1 - (1 - \beta)^m} \leq m \beta, \quad \forall\, \beta > 0,
\]
and the fact that $\rho(\beta) = \rho(0)\, (1 - \mathcal{O}(\beta))$, we obtain
\begin{align*}
	\norm{g_0(\beta)^m - g_0(0)^m} 
	&\lesssim \abs{\rho(\beta)^m - \rho(0)^m} \\
	&\leq \rho(0)^m \abs{\left( \frac{\rho(\beta)}{\rho(0)} \right)^m - 1} \\
	&\leq \rho(0)^m \abs{(1 - \mathcal{O}(\beta))^m - 1} \\
	&\lesssim m\, \abs{a - E}^m\, \beta.
\end{align*}

Now, let $\psi \in \Hol(\Mat_2(\R))$ be a H\"older-continuous function with H\"older constant $\leq 1$. Then
\begin{align*}
\abs{\int \psi\, dg_0(\beta)^m_\ast (\tilde\theta_{\mu,E})^{\ast n} - \int \psi\, dg_0(0)^m_\ast (\tilde\theta_{\mu,E})^{\ast n} } \hspace{-5cm} & \\
&\leq \int  \abs{\psi( g_0(\beta)^m   h)  -  \psi(g_0(0)^m h) }\,  d(\tilde\theta_{\mu,E})^{\ast n} (h)\\
&\leq \int  \norm{ g_0(\beta)^m   h  -   g_0(0)^m h }^p\,  d(\tilde\theta_{\mu,E})^{\ast n} (h)\\
&\lesssim  \norm{ g_0(\beta)^m   -   g_0(0)^m }^p\\
&\lesssim m^p \abs{a-E}^{p\, m} \beta^p.
\end{align*}

Therefore, takin the supremum in $\psi$,
\[
W_p\left((g_0(\beta)^m)_\ast (\tilde{\theta}_{\mu,E})^{\ast n},\, (g_0(0)^m)_\ast (\tilde{\theta}_{\mu,E})^{\ast n} \right)
\lesssim m^p\, \abs{a - E}^{p m}\, \beta^p.
\]

Summing over $n$ and $m$, we obtain
\begin{align*}
	W_p(\tilde{\mu}_{\beta,E,q}, \tilde{\mu}_{0,E,q}) 
	&\lesssim \sum_{n=1}^\infty \sum_{m=1}^\infty m^p\, q^m\, \abs{a - E}^{p m}\, (1 - q)^n\, \beta^p \\
	&\lesssim \beta^p,
\end{align*}
where the series converges provided $0 < p < d$ is chosen small enough so that $q\, \abs{a - E}^p < 1$.

Therefore, by Theorem~\ref{thm non invertible Holder LE }, we conclude:
\begin{align*}
	L_1(\mu_{\lambda, E, q}) 
	&= q \log \lambda + L_1(\tilde{\mu}_{\beta, E, q}) \\
	&= q \log \lambda + L_1(\tilde{\mu}_{0, E, q}) + \mathcal{O}(\beta^p) \\
	&= q \log \lambda + L_1(\tilde{\mu}_{0, E, q}) + \mathcal{O}(\lambda^{-p}).
\end{align*}

This establishes the desired asymptotic formula.
\end{proof}

\begin{remark} The following explicit formula holds
$$L_1(\tilde{\mu}_{0, E, q}) =
 \sum_{n=1}^\infty \sum_{m=1}^\infty   q^m\, (1 - q)^n\, \left[ m  \log \abs{a - E} + \int \log\abs{\langle e_1, g\, e_1\rangle}\, d(\tilde \theta_{\mu,E})^{\ast n}(g)\right]. $$
\end{remark}

\begin{remark}
If, instead of being a Frostman measure, $\mu$ has finite support, then the Lyapunov exponent $L_1(\tilde{\mu}_{0, E, q})$ is discontinuous in $E$ within the asymptotic spectrum (as $\lambda \to \infty$) of the associated Schr\"odinger operators (see~\cite[Example~5.4]{DDKG25}).
In the general case, when $\mu$ exhibits a combination of both behaviors, a mixture of the two phenomena is expected.
\end{remark}

\begin{remark}
Since the Integrated Density of States (IDS), like the Lyapunov exponent, can be represented as an integral with respect to the stationary measure, see~\cite[Theorem 1.1 and Proposition 2.6]{BCDFK}, similar continuity properties are expected to hold for the IDS.
\end{remark}

\bigskip

\subsection{Example 3}
Let $\Sym_m(\R)$ denote the space of symmetric $m \times m$ real matrices. Consider a probability measure $\mu \in \Prob_c(\Sym_m(\R))$ such that
\begin{equation}
	\label{eq Jacobi moment}
	\sup_{E \in \R} \int \frac{1}{|\det(s - E)|^p} \, d\mu(s) \leq C.
\end{equation}

Define the product space $\Omega := \Sym_m(\R)^\Z$ and let the potential $\mathscr{V} : \Omega \to \Sym_m(\R)$ be given by $\mathscr{V}(\omega) := \omega_0$. This potential defines, for each $\omega \in \Omega$, the Jacobi operator
\[
H_{\lambda, \omega} : \ell^2(\Z, \R^m) \to \ell^2(\Z, \R^m),
\quad (H_{\lambda, \omega} \psi)_n := -(\psi_{n+1} + \psi_{n-1}) + \lambda\, \mathscr{V}(\sigma^n \omega)\, \psi_n,
\]
where $\sigma$ denotes the shift on $\Omega$.

Associated with this operator is the family of Schr\"odinger cocycles $A_{\lambda,E} : \Omega \to \Sp(\R^{2m})$ defined by
\[
A_{\lambda,E}(\omega) := \begin{bmatrix}
	\lambda (\mathscr{V}(\omega) - E I) & -I \\ I & 0
\end{bmatrix},
\]
with values in the symplectic group $\Sp(\R^{2m})$.

In this setting, Craig and Simon~\cite{CraigSimon} proved a Thouless formula that relates the sum of the first $m$ Lyapunov exponents
\[
L_1(\wedge^m A_{\lambda,E}) = L_1(A_{\lambda,E}) + \cdots + L_m(A_{\lambda,E})
\]
to the Integrated Density of States (IDS).

As an application of Corollary~\ref{coro main them det moment}, we obtain the following asymptotic formula for this sum.

\begin{proposition}
	Let $\mu \in \Prob_c(\Sym_m(\R))$ be a probability measure satisfying~\eqref{eq Jacobi moment}. Then, there exists $0 < p' < p$ such that as $\lambda \to \infty$,
	\[
	L_1(\wedge^m A_{\lambda,E}) = m \log \lambda + \int_{\Sym_m(\R)} \log |\det(s - E)| \, d\mu(s) + \mathscr{O}(\lambda^{-p'}).
	\]
\end{proposition}





\medskip

\subsection*{Acknowledgments}

  P.D. was partially supported by FCT - Funda\c{c}\~{a}o para a Ci\^{e}ncia e a Tecnologia, through
 the projects   UIDB/04561/2020 and  PTDC/MAT-PUR/29126/2017.\\
 T.G. was supported by FCT - Funda\c{c}\~{a}o para a Ci\^{e}ncia e a Tecnologia, through the projects UI/BD/152275/2021 and by CEMS.UL Ref. UIDP/04561/2020, DOI 10.54499/UIDP/04561/ 2020.

\bigskip

\bigskip
\bibliographystyle{amsplain}
\bibliography{bib}

\end{document}